\documentclass[a4paper,10pt]{amsart}
\usepackage{amssymb}
\usepackage[arrow,matrix]{xy}
\usepackage{hyperref}

\theoremstyle{ams}
\newtheorem{theorem}{Theorem}[section]
\newtheorem{proposition}[theorem]{Proposition}
\newtheorem{lemma}[theorem]{Lemma}

\theoremstyle{definition}
\newtheorem{question}[theorem]{Question}

\newtheorem{definition}[theorem]{Definition}
\newtheorem{remark}[theorem]{Remark}

\newtheorem{example}[theorem]{Example}

\newcommand{\C}{\mathbb{C}}

\newcommand{\R}{\mathbb{R}}
\newcommand{\Z}{\mathbb{Z}}

\newcommand{\CP}{\mathbb{C}P}

\begin{document}
\title[Hard Lefschetz property of symplectic structures]{Hard Lefschetz property of symplectic structures on compact K\"{a}hler manifolds}

\author[Y. Cho]{Yunhyung Cho}
\address{Departamento de matem\'{a}tica, Centro de An\'{a}lise Matem\'{a}tica, Geometria e Sistemas
Din\^{a}micos-LARSYS, Instituto Superior T\'{e}cnico, Av. Rovisco Pais 1049-001 Lisbon, Portugal}
\address{current address : Center for Geometry and Physics, Institute for Basic Science (IBS), Pohang, Republic of Korea 37673}
\email{yhcho@ibs.re.kr}

\thanks{The author was supported by IBS-R003-D1.}
\dedicatory{This paper is dedicated to my wife.}

\date{\today}
\maketitle

\begin{abstract}
    In this paper, we give a new method to construct a compact symplectic manifold which does not satisfy the hard Lefschetz property.
    Using our method, we construct a simply connected compact K\"{a}hler manifold $(M,J,\omega)$ and a symplectic form $\sigma$ on $M$ which does not satisfy the hard Lefschetz property, but is symplectically deformation equivalent to the K\"{a}hler form $\omega$.
    As a consequence, we can give an answer to the question posed by Khesin and McDuff as follows.
    According to symplectic Hodge theory, any symplectic form $\omega$ on a smooth manifold $M$ defines \textit{symplectic harmonic forms} on $M$.
    In \cite{Yan}, Khesin and McDuff posed a question whether there exists a path of symplectic forms $\{ \omega_t \}$ such that the dimension $h^k_{hr}(M,\omega)$ of the space of \textit{symplectic harmonic $k$-forms} varies along $t$.
    By \cite{Yan} and \cite{Ma}, the hard Lefschetz property holds for $(M,\omega)$ if and only if $h^k_{hr}(M,\omega)$ is equal to the Betti number $b_k(M)$ for all $k>0$.
    Thus our result gives an answer to the question.
    Also, our construction provides an example of compact K\"{a}hler manifold whose K\"{a}hler cone is properly contained in the symplectic cone (c.f. \cite{Dr}).
\end{abstract}

\section{introduction}

For any compact K\"{a}hler manifold $(M,\omega,J)$ of complex dimension $n$, the \textit{hard Lefschetz theorem} states that
\begin{equation}\label{equation : HLP}
    \begin{array}{ccccc}
    [\omega]^{n-k} & : & H^k(M;\R) & \longrightarrow & H^{2n-k}(M;\R) \\
                   &   &   \alpha  & \mapsto         & \alpha \cup [\omega]^{n-k} \\
    \end{array}
\end{equation}
is an isomorphism for every $k=0,1,\cdots,n$.
Now, let us consider a compact symplectic manifold $(M,\omega)$ of real dimension $2n$. Then it is natural to ask whether the hard Lefschetz theorem holds for $\omega$, but it turned out that the hard Lefschetz theorem does not hold in general.

We say that a symplectic form $\omega$ is of \textit{hard Lefschetz type} if the map $[\omega]^{n-k}$ in (\ref{equation : HLP}) is an isomorphism for every $k = 0,1,\cdots,n$, and we say $\omega$ is of \textit{non-hard Lefschetz type} otherwise.
In this paper, we consider the following.

\begin{question}\label{question : non hard Lefschetz Kahler}
    Let $(M,\omega,J)$ be a compact K\"{a}hler manifold. Then is it possible that $M$ admits a symplectic form $\sigma$ of non-hard Lefschetz type?
\end{question}

The reason why we consider Question \ref{question : non hard Lefschetz Kahler} is as follows. Although there are many examples of compact symplectic manifolds
of non-hard Lefschetz type, not all of them are homotopy equivalent to a K\"{a}hler manifold.
The simplest example of non-hard Lefschetz type is a compact symplectic manifold $(M,\omega)$ such that the $(2k+1)$-th Betti number $b_{2k+1}(M)$ is odd for some $k \in \Z_{\geq 0}$.
Such a manifold does not admit a K\"{a}hler structure by Hodge symmetry.
Gompf also constructed a family of compact symplectic manifolds of non-hard Lefschetz type as follows.

\begin{theorem}\cite[Theorem 7.1]{Gom}\label{theorem : Gompf}
    For any finitely presentable group $G$ and any integers $n \geq 3$ and $b \geq 0$, there exists a $2n$-dimensional compact manifold $M$ such that
    \begin{itemize}
        \item $M$ admits a symplectic structure,
        \item $\pi_1(M) \cong G$,
        \item $b_i(M) \geq b$ for $2 \leq i \leq n-2$, and
        \item $M$ does not admit any symplectic form of hard Lefschetz type.
    \end{itemize}
\end{theorem}
In particular, the last condition in Theorem \ref{theorem : Gompf} implies that $M$ in Theorem \ref{theorem : Gompf} is not homotopy equivalent to any K\"{a}hler manifold. 
Fern\'{a}ndez and Mu\~{n}oz \cite[Remark 3.3]{FM} constructed a simply connected non-formal symplectic manifold of non-hard Lefschetz type. Note that the non-formality implies that their example is not homotopy equivalent to any K\"{a}hler manifold. For the case of nilmanifolds, Benson and Gordon \cite[Theorem A]{BG} proved that a compact symplectic nilmanifold satisfies the hard Lefschetz property if and only if it is isomorphic to a torus. Also, Hasegawa \cite{H} proved that a compact symplectic nilmanifold is formal if and only if it is isomorphic to a torus. Consequently, a compact symplectic nilmanifold is of hard Lefschetz type if and only if it is K\"{a}hler, i.e., every compact nilmanifold of non-hard Lefschetz type is not homotopy equivalent to any K\"{a}hler manifold. In this point of view, we give the first example of compact symplectic manifold of non-hard Lefschetz type which is homotopy equivalent to some K\"{a}hler manifold, thereby giving an answer to Question \ref{question : non hard Lefschetz Kahler} as follows.
\begin{theorem}\label{theorem : main}
    There exists a compact K\"{a}hler manifold $(X,\omega,J)$ with $\dim_{\C} X = 3$ such that
    \begin{enumerate}
        \item $X$ is simply connected,
        \item $H^{2k+1}(X) = 0$ for every integer $k \geq 0$,
        \item $X$ admits a symplectic form $\sigma \in \Omega^2(X)$ of non-hard Lefschetz type, and
        \item $\sigma$ is deformation equivalent to the K\"{a}hler form $\omega$.
    \end{enumerate}
\end{theorem}

There are three immediate applications of Theorem \ref{theorem : main}. Firstly, the condition (4) in Theorem ~\ref{theorem : main} implies that the hard Lefschetz property is not an invariant property under symplectic deformations. Secondly, it provides an example of a compact manifold whose K\"{a}hler cone is non-empty and properly contained in the symplectic cone. For a given manifold $X$, the \textit{K\"{a}hler cone} $\mathcal{K}(X)$ is a subset of $H^2(X;\R)$ such that each element of $\mathcal{K}(X)$ can be represented by some K\"{a}hler form on $X$. Similarly, the \textit{symplectic cone} $\mathcal{S}(X)$ is defined as a subset of $H^2(X;\R)$ such that each element of $\mathcal{S}(X)$ can be represented by some symplectic form on $X$. Such examples ($\emptyset \neq \mathcal{K}(X) \subsetneq \mathcal{S}(X)$) were studied by Dr\u{a}ghici \cite{Dr} and Li-Usher \cite{LU} in four dimensional case.
Thirdly, Theorem \ref{theorem : main} gives an answer to a question of Khesin and McDuff in the simply connected case as follows.

Let $(M,\omega)$ be a $2n$-dimensional compact symplectic manifold and we denote by $\Omega^k(M)$ the set of all $k$-forms on $M$.
Then we can define the \textit{symplectic Hodge star operator $*_{\omega} : \Omega^k(M) \rightarrow \Omega^{2n-k}(M)$ with respect to $\omega$} satisfying
$$\alpha \wedge *_{\omega} \beta = \omega(\alpha,\beta)\frac{\omega^n}{n!}$$
for every $\alpha,\beta \in \Omega^k(M)$. Here we regard $\omega$ as the extension of the symplectic form on $M$ to $\Omega^k(M)$ after identifying $TM$ with $T^*M$ via
\begin{displaymath}
    \begin{array}{ccc}
        TM & \rightarrow & T^*M \\
        v & \mapsto & \omega(v,\cdot).\\
    \end{array}
\end{displaymath}
We say that $\alpha \in \Omega^k(M)$ is \textit{symplectic harmonic} if $d\alpha = \delta\alpha = 0$ where $\delta\alpha = (-1)^k*_{\omega}d*_{\omega}\alpha$.
Let $H^*_{hr}(M,\omega)$ be the set of all cohomology classes which can be represented by \textit{symplectic harmonic forms} with respect to $\omega$ and we denote $$h_k(M,\omega) := \dim H^k_{hr}(M,\omega).$$
Brylinski \cite{Bry} proved that if $\omega$ is K\"{a}hler, then every cohomology class has a symplectic harmonic representative so that $h_k(M,\omega) = b_k(M)$ for every $k \geq 0$.
Also he conjectured that his theorem can be extended to any compact symplectic manifold, but Mathieu \cite{Ma} and Yan \cite{Yan} disproved Brylinski's conjecture independently.
\begin{theorem}\cite{Ma}\cite{Yan}
    Let $(M,\omega)$ be a compact symplectic manifold. Then $h_k(M,\omega) = b_k(M)$ for every $k \in \Z_{\geq 0}$if and only if $\omega$ is of hard Lefschetz type.
\end{theorem}
As in \cite[Section 4]{Yan}, Khesin and McDuff posed the question on the existence of continuous family of symplectic forms $\{\omega_t\}$ on a closed manifold $M$ such that $h_k(M,\omega_t) = \dim H^k_{hr}(M,\omega_t)$ varies with respect to $t$. As in \cite{Yan} and \cite{IRTU}, there are some examples of symplectic manifolds such that $h_k(M,\omega_t)$ varies along $t$, but any of their examples is neither simply connected nor homotopy equivalent to any K\"{a}hler manifold. Hence Theorem \ref{theorem : main} provides the first simply connected example with varying $h_t(M,\omega_t)$ such that $M$ is homotopy equivalent to some K\"{a}hler manifold.

In fact, the homeomorphism type of our manifold given in Theorem \ref{theorem : main} is very simple. It is a two-sphere bundle over some four manifold with $b_2^+ \geq 2$ (K3-surface, for example). We sketch the construction as follows.
Let us consider a compact symplectic four manifold $(N,\sigma)$ and an integral cohomology class $e \in H^2(N;\Z)$. Then there is a complex line bundle $\xi$ over $N$ such that $c_1(\xi) = e$. For the associate principal $S^1$-bundle $\mathbb{S}(\xi)$, let $M(N,e) = \mathbb{S}(\xi) \times [-\epsilon,+\epsilon]$ with a symplectic form $\omega_{\sigma} = \pi^*\sigma + d(t\cdot \theta)$ where $\pi : \mathbb{S}(\xi) \rightarrow N$ is the quotient map by the $S^1$-action, $t$ is the parameter for $[-\epsilon,+\epsilon]$, and $\theta$ is any connection 1-form on $\mathbb{S}(\xi)$. Then the induced action on $M(N,e)$ is free and Hamiltonian with a moment map
\begin{displaymath}
\begin{array}{cccc}
          H : & M(N,e) & \longrightarrow & [-\epsilon, +\epsilon] \subset \R\\ 
              &   (z, t)&   \mapsto & t \\
\end{array}
\end{displaymath}
whose maximum and minimum are diffeomorphic to $\mathbb{S}(\xi)$. If we apply the symplectic cut method \cite{Ler} to $M(N,e)$ along the extremum, then the induced space, which is denoted by $\widetilde{M}(N,e)$, is compact without boundary and admits the reduced symplectic form $\widetilde{\omega}_{\sigma}$. The space $(\widetilde{M}(N,e), \widetilde{\omega}_{\sigma})$ (obtained by a symplectic quotient $M(N,e) \times \C^2 /\!\!/ T^2$, see Section 2) is our main object. In Section 2, we will prove the following.

\begin{proposition}\label{proposition : sphere bundle}
    $\widetilde{M}(N,e)$ is diffeomorphic to $\mathbb{P}(\xi \oplus \underline{\C})$ where $\underline{\C}$ is the trivial line bundle over $N$.
    In particular, if $(N,\sigma,J)$ is a compact K\"{a}hler manifold and $e \in H^2(N;\R)$ is of $(1,1)$-type, then there exists another K\"{a}hler form $\sigma'$ on $N$ such that
    \begin{itemize}
        \item $\widetilde{\omega}_{\sigma'}$ is a K\"{a}hler form on $\widetilde{M}(N,e)$, and
        \item $\widetilde{\omega}_{\sigma'}$ is deformation equivalent to $\widetilde{\omega}_{\sigma}$.
    \end{itemize}
\end{proposition}

Note that two symplectic forms $\sigma$ and $\gamma$ on $N$ are called \textit{deformation equivalent} if there is a path $\{ \sigma_t \}_{0 \leq t \leq 1}$ of symplectic forms such that $\sigma_0 = \sigma$ and $\sigma_1 = \gamma$.
The main difficulty in proving Proposition \ref{proposition : sphere bundle} is that there is no guarantee that $\omega_{\sigma}$ is a K\"{a}hler form on $M(N,e)$ even though $(N,\sigma,J)$ is K\"{a}hler. However we will show that by perturbing $\sigma$ we can obtain a new K\"{a}hler form $\sigma'$ on $N$ such that $(M(N,e), \omega_{\sigma'})$ is K\"{a}hler. We will also see that our space $(\widetilde{M}(N,e), \widetilde{\omega}_{\sigma'})$ can be obtained by a K\"{a}hler quotient so that the reduced symplectic form $\widetilde{\omega}_{\sigma'}$ is a K\"{a}hler form. See Section 2 for more details.

Now, suppose that $(N,\sigma,J)$ is a K\"{a}hler manifold and $e \in H^2(N;\Z)$ is of $(1,1)$-type such that $(\widetilde{M}(N,e), \widetilde{\omega}_{\sigma})$ is K\"{a}hler.
Then the hard Lefschetz property for $\widetilde{\omega}_{\sigma}$ is automatically satisfied. On the other hand, if we choose another symplectic form $\tau$ which is \textbf{NOT} K\"{a}hler with respect to $J$, then the result is completely different.
In fact, the reduced symplectic form $\widetilde{\omega}_{\tau}$ on $\widetilde{M}(N,e)$ may not satisfy even the hard Lefschetz property in this case.
To check whether the reduced symplectic form $\widetilde{\omega}_{\tau}$ is of hard Lefschetz type or not, we will study our space $(\widetilde{M}(N,e),\widetilde{\omega}_{\tau})$ in a more general setting. Note that since the action is free and Hamiltonian on $M(N,e)$, the induced circle action on $(\widetilde{M}(N,e), \widetilde{\omega}_{\tau})$ is semi-free and Hamiltonian and the fixed point set consists of two copies of $\mathbb{S}(\xi) / S^1 \cong N$. Such manifold is called a \textit{simple Hamiltonian $S^1$-manifold} (See \cite{HH}).

\begin{definition}\label{definition : simple Hamiltonian}
    Let $(M,\omega)$ be a smooth compact symplectic manifold and let $S^1$ be the unit circle group which acts on $(M,\omega)$ in a Hamiltonian fashion.
    We call $(M,\omega)$ a \textit{simple Hamiltonian $S^1$-manifold} if the fixed point set $M^{S^1}$ consists of two connected fixed components.
\end{definition}

Now, assume that $(M,\omega)$ is a six-dimensional simple Hamiltonian $S^1$-manifold. By scaling the symplectic structure $\omega$, we may assume that there is a moment map $H : M \rightarrow [0,1]$. By definition, we have two fixed components $Z_{\min} = H^{-1}(0)$ and $Z_{\max} = H^{-1}(1)$. Since any fixed component of any Hamiltonian Lie group action is a symplectic submanifold of $M$, a dimension of $Z_{\min}$ ($Z_{\max}$, respectively) is zero, two, or four.

Firstly, let us consider the case where $\dim Z_{\min} = \dim Z_{\max} = 4$, in which we are particularly interested. In this case, we may identify $Z_{\min}$ with $Z_{\max}$ as follows. The normal bundle of $Z_{\min}$ is a complex line bundle over $Z_{\min}$ with the induced circle action on each fiber $\C$ as a rotation. Hence any level set $H^{-1}(t)$ near the minimum $Z_{\min}$ is a principal $S^1$-bundle over $Z_{\min}$ so that the reduced space $H^{-1}(t) / S^1$ is diffeomorphic to $Z_{\min}$. Similarly, a reduced space near the maximum $Z_{\max}$ is diffeomorphic to $Z_{\max}$. Since there is no critical submanifold except for $Z_{\min}$ and $Z_{\max}$, we may identify $Z_{\min}$ with $Z_{\max}$ along the gradient flow with respect to $H$. Thus we may compare the induced symplectic form $\omega|_{Z_{\min}}$ on $Z_{\min}$ with $\omega|_{Z_{\max}}$ on $Z_{\max}$ via the identification described above. The following proposition gives the complete description of the hard Lefschetz property of $\omega$ in terms of $\omega|_{\min}$ and $\omega|_{\max}$.

\begin{proposition}\label{proposition : if and only if}
    Let $(M,\omega)$ be a six-dimensional simple Hamiltonian $S^1$-manifold with a moment map $H : M \rightarrow [0,1]$. Assume that all fixed components are of dimension four so that $Z_{\min} = H^{-1}(0) \cong H^{-1}(1) = Z_{\max}$. Then $(M,\omega)$ satisfies the hard Lefschetz property if and only if
    \begin{itemize}
        \item $(Z_{\min}, \omega_0 + \omega_1)$ satisfies the hard Lefschetz property, and
        \item $[\omega_0] \cdot [\omega_1] \neq 0$ in $H^4(Z_{\min};\R)$,
    \end{itemize}
    where $\omega_0 = \omega|_{Z_{\min}}$ and $\omega_1 = \omega|_{Z_{\max}}$ respectively.
\end{proposition}

Therefore, to prove Theorem \ref{theorem : main}, it is enough to find a K\"{a}hler surface $(N,\sigma,J)$, an integral class $e \in H^{1,1}(N)$, and a non-K\"{a}hler (with respect to $J$) symplectic form $\tau$ on $N$ such that $\widetilde{\omega}_{\tau}$ violates the condition in Proposition \ref{proposition : if and only if} by Proposition \ref{proposition : sphere bundle} (See Section 5).

Now, let us consider the remaining case, i.e., the case where there is a fixed component of dimension less than four. Even though this case is not relevant to our main theorem,
we present it for completeness.
\begin{proposition}\label{proposition : off-topic}
    Let $(M,\omega)$ be a six-dimensional simple Hamiltonian $S^1$-manifold. If there is a fixed component of dimension less than four, then $(M,\omega)$ satisfies the hard Lefschetz property.
\end{proposition}

This paper is organized as follows.
In Section 2, we give the detail of the construction of $(\widetilde{M}(N,e),\widetilde{\omega}_{\sigma})$ described above.
Also we give the proof of Proposition \ref{proposition : sphere bundle}.
In Section 3, we briefly review equivariant cohomology theory and the Atiyah-Bott-Berline-Vergne localization theorem which will be used for checking the hard Lefschetz property of $\widetilde{\omega}_{\sigma}$.
In Section 4, we give the proof Proposition \ref{proposition : if and only if} and Proposition \ref{proposition : off-topic}.
Finally, in Section 5, we give the prove of our main result Theorem \ref{theorem : main} and provide several examples of compact symplectic manifolds of non-hard Lefschetz type.

\section{Construction}

Let $(N,\sigma)$ be a compact symplectic manifold and let $e \in H^2(N;\Z)$ be an integral cohomology class. Let $\xi$ be the complex line bundle over $N$ such that $c_1(\xi) = e$.
Then the projective bundle $\mathbb{P}(\xi \oplus \underline{\C})$ is a two-sphere bundle over $N$ where $\underline{\C}$ is the trivial complex line bundle over $N$.

To make $\mathbb{P}(\xi \oplus \underline{\C})$ to be symplectic, we choose a different way of constructing $\mathbb{P}(\xi \oplus \underline{\C})$. In fact, $\mathbb{P}(\xi \oplus \underline{\C})$ can be obtained as the symplectic quotient by $T^2$ as follows.
Let $P$ be the principal $S^1$-bundle over $N$ whose first Chern class $c_1(P)$ is $e$, i.e., its associated complex line bundle is $\xi$. For $\epsilon > 0 $, let $$M(N,e,\epsilon) = P \times (-\epsilon, +\epsilon)$$ and let $$\omega_{\sigma} = \pi^*\sigma + d(r\cdot \theta) = \pi^*\sigma + dr\wedge \theta + r\cdot d\theta$$ be a 2-form on $M(N,e,\epsilon)$ where $\pi : P \rightarrow N$ is the quotient map by the $S^1$-action, $\theta$ is any connection 1-form on $P$, and $r$ is the parameter of $(-\epsilon,+\epsilon)$. Then $\omega_{\sigma}$ is closed and non-degenerate on $M(N,e,\epsilon)$ for sufficiently small $\epsilon$, since $\pi^*\sigma + dr\wedge \theta$ is constant along $r \in (-\epsilon, \epsilon)$ and non-degenerate everywhere on $M(N,e,\epsilon)$ so that $r\cdot d\theta$ does not affect the non-degeneracy of $\omega_{\sigma}$ for a sufficiently small $r$.

Now, suppose $(N,\sigma,J)$ is a compact K\"{a}hler manifold with $e \in H^{1,1}(N;\Z)$. Then $\xi$ becomes a holomorphic line bundle and the total space of $\xi$ admits an integrable almost complex structure $I$ such that the projection map $\mathrm{pr} : (\xi,I) \rightarrow (N,J)$ is holomorphic. Unfortunately, there is no guarantee that $\omega_{\sigma}$ is K\"{a}hler on $M(N,e,\epsilon) = P \times (-\epsilon, +\epsilon)$ since $d(r\cdot \theta)$ may not be of $(1,1)$-type with respect to $I$.
However we may perturb $\sigma$ into another K\"{a}hler form $\sigma'$ which makes $\omega_{\sigma'}$ to be a K\"{a}hler form on $M(N,e,\epsilon)$ as we see below.

\begin{proposition}\label{proposition : symplectization is Kaehler}
    Suppose that $(N,\sigma,J)$ is a compact K\"{a}hler manifold with $e \in H^{1,1}(N;\Z)$. Then there exists another K\"{a}hler form $\sigma'$  on $N$ such that
     \begin{itemize}
        \item $\omega_{\sigma'}$ is a K\"{a}hler form on $M(N,e,\epsilon)$ for a sufficiently small $\epsilon > 0$, and
        \item $\sigma'$ is deformation equivalent to $\sigma$.
     \end{itemize}
\end{proposition}

\begin{proof}
    Suppose that $(M,\omega)$ is a Hamiltonian $S^1$-manifold with a moment map $\mu : M \rightarrow \R$ and let $P$ be a manifold with a free $S^1$-action such that $P$ is $S^1$-equivariantly diffeomorphic to a level set $\mu^{-1}(r)$ for some regular value $r \in \R$. Let $\sigma$ be the reduced symplectic form on the quotient space $\mu^{-1}(r) / S^1$.
    Then for a sufficiently small $\epsilon > 0$, the maps
    \begin{displaymath}
        \begin{array}{ccccl}
                i_1 & : & P & \hookrightarrow & (P \times (-\epsilon,+\epsilon), \omega_{\sigma})   \\
                    &   & p & \mapsto         & (p,0) \\
        \end{array}
    \end{displaymath}
    and
    \begin{displaymath}
        \begin{array}{ccccc}
                i_2 & : & P \cong \mu^{-1}(r) & \hookrightarrow & (M,\omega)   \\
        \end{array}
    \end{displaymath}
    are both $S^1$-equivariant co-isotropic embeddings and $$i_1^*\omega_{\sigma} = (\pi^*\sigma + r \cdot d\theta + dr \wedge \theta)|_{P \times \{0\}} = \pi^*\sigma= i_2^* \omega$$ where $\pi : P \rightarrow P / S^1 \cong \mu^{-1}(r) / S^1$ is the quotient map.
    By the \textit{equivariant co-isotropic embedding theorem} \cite{CdS}(footnote in p.193), there exists a tubular neighborhood $\mathcal{U}$ of $P$ in $M$ and an $S^1$-equivariant symplectomorphism
    $$\phi : (P \times (-\epsilon, \epsilon), \omega_{\sigma}) \rightarrow (\mathcal{U}, \omega|_{\mathcal{U}})$$
    for a sufficiently small $\epsilon > 0$.
    This means that if $(M,\omega,I)$ is K\"{a}hler, then $(\mathcal{U}, \omega|_{\mathcal{U}}, I|_{\mathcal{U}})$ is also K\"{a}hler so that $\omega_{\sigma}$ is a K\"{a}hler form on $P \times (-\epsilon, \epsilon)$ with respect to $\phi^*I|_{\mathcal{U}}$.

    Now, let $(N,\sigma,J)$ be a compact K\"{a}hler manifold with $e \in H^{1,1}(N;\Z)$. Since any reduced symplectic form obtained by a K\"{a}hler quotient is K\"{a}hler, it is enough to show that there exists some K\"{a}hler manifold $(M,\omega,I)$ equipped with a holomorphic circle action with a moment map $\mu : M \rightarrow \R$ such that some level set $\mu^{-1}(r)$ is $S^1$-equivariantly diffeomorphic to the associate principal $S^1$-bundle $P$ with $c_1(P) = e$, and the reduced K\"{a}hler form on $\mu^{-1}(r) / S^1$ is deformation equivalent to $\sigma$. To show this, we use the idea of Proposition 3.18 in \cite{Vo} as follows.

    Let $\xi$ be the holomorphic line bundle over $N$ such that $c_1(\xi) = e$.
    Also, let $\underline{\C} = N \times \C$ be the trivial line bundle over $N$ with the holomorphic structure induced by the integrable almost complex structure $J$ on $N$ and the standard complex structure on $\C$. Then the direct sum of two line bundles $$E :=\xi \oplus \underline{\C}$$ admits an induced holomorphic structure. Since each line bundle has a holomorphic $\C^*$-action as a fiber-wise scalar multiplication, $E$ admits a holomorphic $(\C^*)^2$-action compatible with the holomorphic structure on $E$.

    Firstly, we will construct a K\"{a}hler form on the projectivized bundle $\mathbb{P}(E)$ as described in \cite{Vo}. Let $\mathcal{O}_{\mathbb{P}(E)}(-1)$ be the tautological line bundle over $\mathbb{P}(E) = (E-0_E) / \C^*$ where $0_E$ means the zero section of $E$ and $\C^*$ is the diagonal subtorus of $(\C^*)^2$. Then $\mathbb{P}(E)$ is a $\mathbb{P}^1$-bundle over $N$. Let $h$ be any hermitian metric on $\mathcal{O}_{\mathbb{P}(E)}(-1)$. Then there exists a unique connection $\bigtriangledown$ (called the Chern connection) such that the corresponding curvature form $\Theta_E$ (called the Chern curvature) is purely imaginary and of $(1,1)$-type. If we restrict $\Theta_E$ onto any fiber $F \cong \mathbb{P}^1$ of $\mathbb{P}(E)$, then it is nothing but the Chern curvature of $\mathcal{O}_{\mathbb{P}^1}(-1)$ so that $-\sqrt{-1}\Theta_E|_F$ is the Fubini-Study form on $\mathbb{P}^1$. In particular, $\sqrt{-1}\Theta_E$ is non-degenerate on each fiber of $\mathbb{P}(E)$. Let
    $$ \Omega := C \cdot \mathrm{pr}^*\sigma - \sqrt{-1}\Theta_E $$
    where $C$ is a constant and $\mathrm{pr} : \mathbb{P}(E) \rightarrow N$ is the projection map which is holomorphic. Then $\Omega$ is obviously a closed $(1,1)$-form, and is non-degenerate on each fiber. Since $N$ is compact, we can take $C$ large enough so that $\Omega$ is also non-degenerate along the horizontal direction. Therefore $\Omega$ is a K\"{a}hler form on $\mathbb{P}(E)$.

    Now, we construct our desired K\"{a}hler form $\sigma'$ on $N$ as follows. Let us start with a hermitian metric $h$ on $E$ invariant under the $T^2$-action
    where $T^2$ is the compact subtorus of $(\C^*)^2$. (Such metric can be obtained from any hermitian metric by averaging with the Haar measure on $T^2$). Since $\mathcal{O}_{\mathbb{P}(E)}(-1)$ is the blow-up of $E$ along the zero-section $0_E$, there is a natural identification between $\mathcal{O}_{\mathbb{P}(E)}(-1) - 0_{{\mathbb{P}(E)}(-1)}$ and $E-0_E$.
    Hence there is an induced hermitian metric on $\mathcal{O}_{\mathbb{P}(E)}(-1)$, which we still call it $h$. Since $\mathbb{P}(E) = (E-0_E) / \C^*$, there is a residual $\C^*$-action on $\mathbb{P}(E)$ induced by the $\C^*$-action on $E = \xi \oplus \underline{\C}$ such that $t \cdot (n, z) = (t\cdot n, z)$ for $t \in \C^*$. Then the induced action of $S^1 \subset \C^*$ is holomorphic by our assumption, and preserves $\Omega$ since $h$ is $S^1$-invariant so that the corresponding connection $\bigtriangledown$ and the curvature form $\Theta_E$ is also $S^1$-invariant. Hence the $S^1$-action on $(\mathbb{P}(E), \Omega)$ is holomorphic Hamiltonian. Note that the fixed point set consists of two components
    $\mathbb{P}(\xi \oplus 0) \cong N$ and $\mathbb{P}(0 \oplus \underline{\C}) \cong N$. The restriction of $\Omega$ onto $\mathbb{P}(0 \oplus \underline{\C})$ is nothing but $C \cdot \sigma$ since the restriction of $\mathcal{O}_{\mathbb{P}(E)}(-1)$ onto $\mathbb{P}(0 \oplus \underline{\C})$ is just $\underline{\C}$ and the Chern curvature form vanishes on the trivial line bundle. Hence if $\mu : \mathbb{P}(E) \rightarrow \R$ is a moment map and if we take any regular value $r \in \R$ near the critical value $\mu(\mathbb{P}(0 \oplus \underline{\C})$, then the reduced symplectic form denoted by $\sigma'$ is K\"{a}hler on the K\"{a}hler quotient $\mu^{-1}(r) / S^1$ and is deformation equivalent to $C \cdot \sigma$.
    Consequently, $\sigma'$ is deformation equivalent to $\sigma$.

\end{proof}

From now on, we will construct our main object $\widetilde{M}(N,e,\epsilon)$ by taking a symplectic quotient of a certain Hamiltonian $T^2$-manifold as follows.
Let us assume that $\epsilon > 0$ is small enough so that $(M(N,e,2\epsilon),\omega_{\sigma})$ is a symplectic manifold.
Then the induced circle action on $(M(N,e,2\epsilon),\omega_{\sigma})$ satisfies $i_X\omega_{\sigma} = dr$ so that the action is Hamiltonian with respect to the moment map $H(p,r) = r$ for $p \in P$ and $r \in I_{2\epsilon}$. Now, we will apply Lerman's symplectic cutting\footnote{In fact, Lerman used $S^1$-action in his paper \cite{Ler}. But in our case, we use $T^2$-action instead of $S^1$.} \cite{Ler} to $(M(N,e,2\epsilon), \omega_{\sigma})$ as follows.

Consider the following symplectic manifold
$$(M(N,e,2\epsilon) \times \C^2, \omega_{\sigma} + \frac{1}{2} \sum_{j=1}^2 \frac{1}{\sqrt{-1}} dz_j \wedge d\bar{z}_j)$$ with a Hamiltonian $T^3$-action which is given by $$(t_1, t_2, t_3) \cdot (p, r, z_1, z_2) = (t_1t_2t_3p, r, t_2^{-1}z_1, t_3z_2)$$ for $p \in P$, $r \in I_{2\epsilon}$, $(z_1,z_2) \in \C^2 $, and $(t_1,t_2,t_3) \in T^3$. Let $T_2 = 1 \times S^1 \times S^1 \subset T^3$ be the 2-dimensional subtorus of $T^3$. Then a moment map $\mu$ for the $T_2$-action is given by $$\mu(p,r,z_1,z_2) = (H(p,r) - |z_1|^2, H(p,r) + |z_2|^2) = (r - |z_1|^2, r + |z_2|^2) $$ for $(p,r) \in M(N,e,2\epsilon)$ and $(z_1, z_2) \in \C^2$.
Note that a point $(p,r,z_1,z_2)$ has a non-trivial stabilizer for the $T_2$-action if and only if $z_1 = z_2 = 0$, which implies that the $T_2$-action is free on the level set $\mu^{-1}(a,b)$ for $a \neq b$.
Let $\widetilde{M}(N,e,2\epsilon)_{-\epsilon, \epsilon}$ be the symplectic quotient of the level set $\mu^{-1}(-\epsilon,\epsilon)$, i.e.,
$$\widetilde{M}(N,e,2\epsilon)_{-\epsilon, \epsilon} := \mu^{-1}(-\epsilon,\epsilon) / T_2$$
with the reduced symplectic form $\widetilde{\omega}_{\sigma}$. Then we can prove the following proposition (Proposition \ref{proposition : sphere bundle}) which is our main goal of this section.

\begin{proposition}[Proposition \ref{proposition : sphere bundle}]
    The symplectic quotient $\widetilde{M}(N,e,2\epsilon)_{-\epsilon, \epsilon}$ is diffeomorphic to $\mathbb{P}(\xi \oplus \underline{\C})$ where $\underline{\C}$ is the trivial line bundle over $N$.
    In particular, if $(N,\sigma,J)$ is a compact K\"{a}hler manifold and $e \in H^2(N;\R)$ is of $(1,1)$-type, then there exists another K\"{a}hler form $\sigma'$ on $N$ such that
    \begin{itemize}
        \item $\widetilde{\omega}_{\sigma'}$ is a K\"{a}hler form on $\widetilde{M}(N,e,2\epsilon)_{-\epsilon, \epsilon}$, and
        \item $\widetilde{\omega}_{\sigma'}$ is deformation equivalent to $\widetilde{\omega}_{\sigma}$.
    \end{itemize}
\end{proposition}

\begin{proof}

    Firstly, we observe that $$\mu^{-1}(-\epsilon, \epsilon) =  \{(p,r,z_1,z_2) \in P \times I_{2\epsilon} \times \C^2 ~|~ r + \epsilon = |z_1|^2, -r + \epsilon = |z_2|^2 \}.$$
    Since any point $(p,r,z_1,z_2) \in \mu^{-1}(-\epsilon,\epsilon)$ satisfies $|z_1|^2 + |z_2|^2 = 2\epsilon$ and $r$ is determined by the values $z_1$ and $z_2$ automatically, there is a $T_2$-equivariant diffeomorphism
    \begin{displaymath}
        \begin{array}{lllll}
             \phi & : & \mu^{-1}(-\epsilon,\epsilon) & \rightarrow & P \times S^3 \subset P \times \C^2 \\
                  &   &    (p,r,z_1,z_2)             & \mapsto     & (p, z_1, z_2) \\
        \end{array}
    \end{displaymath}
    where $S^3$ is a sphere in $\C^2$ of radius $\sqrt{2\epsilon}$, and the $T_2$-action on $P \times S^3$ is given by
    \begin{equation}
        (t_2,t_3)\cdot (p,z_1,z_2) = (t_2t_3p, t_2^{-1}z_1, t_3z_2)
    \end{equation}
    for $(p,z_1,z_2) \in P \times S^3$ and $(t_2,t_3) \in T_2$.
    Note that the $T_2$-action on each space is free, hence $\phi$ induces a diffeomorphism $\widetilde{\phi}$ between the quotient spaces
    $$ \widetilde{\phi} : \widetilde{M}(N,e,2\epsilon)_{-\epsilon, \epsilon} = \mu^{-1}(-\epsilon,\epsilon)/T_2 \rightarrow P \times_{T_2} S^3. $$
    Note that $\xi \cong P \times_{S^1} \C$ where $S^1$ acts on $P \times \C$ by $t \cdot (p,z) = (t\cdot p, t^{-1}z)$ so that we have $\xi \oplus \underline{\C} \cong (P \times_{S^1} \C) \times \C$. Then the projectivization $\mathbb{P}(\xi \oplus \underline{\C})$ is the quotient $(\xi \oplus \underline{\C}) / \C^*$ which is equivalent to $(P \times \C \times \C) / S^1 \times \C^*$ where $(t,w) \in S^1 \times \C^*$ acts on
    $(p,z_1,z_2) \in P \times \C \times \C$ by $$(t,w)\cdot(p,z_1,z_2) = (tp,t^{-1}wz_1,wz_2).$$
    Also, the quotient $(P \times \C \times \C) / S^1 \times \C^*$ is equivalent to the quotient $(P \times S^3) / S^1 \times S^1$ by the compact torus $S^1 \times S^1 \subset S^1 \times \C^*$. If we substitute $t$ with $tw$, then the action is exactly the same as the $T_2$-action on $P \times S^3$ in (3), which completes the proof of the first statement.

    For the second statement, if $(N,\sigma,J)$ is a compact K\"{a}hler manifold with $e \in H^{1,1}(N;\Z)$, then there exists another K\"{a}hler form $\sigma'$ on $N$ such that
    \begin{itemize}
        \item $(M(N,e,2\epsilon),\omega_{\sigma'})$ is K\"{a}hler, and
        \item $\sigma'$ is deformation equivalent to $\sigma$
    \end{itemize}
    by Proposition \ref{proposition : symplectization is Kaehler}. Hence the $T_2$-action on $ M(N,e,2\epsilon) \times \C^2$ is holomorphic and Hamiltonian with respect to the K\"{a}hler form $\omega_{\sigma'} + \frac{1}{2} \sum_{j=1}^2 \frac{1}{\sqrt{-1}} dz_j \wedge d\bar{z}_j$ so that the reduced symplectic form $\widetilde{\omega}_{\sigma'}$ is K\"{a}hler on $\widetilde{M}(N,e,2\epsilon)_{-\epsilon, \epsilon}$.
    Finally, we can conclude that $\widetilde{\omega}_{\sigma}$ is deformation equivalent to $\widetilde{\omega}_{\sigma'}$ by Lemma \ref{lemma : symplectic deformation} below.
\end{proof}

In Section 5, we will find a suitable compact K\"{a}hler surface $(N,\sigma,J)$ and $e \in H^{1,1}(N)$ with respect to $J$ such that our symplectic manifold $(\widetilde{M}(N,e,2\epsilon)_{-\epsilon, \epsilon}, \widetilde{\omega}_{\sigma})$ is also K\"{a}hler by Proposition \ref{proposition : sphere bundle}. Then $\widetilde{\omega}_{\sigma}$ satisfies the hard Lefschetz property automatically. But if we take another symplectic form $\tau$ on $N$ which is not K\"{a}hler with respect to $J$, then we will see in Section 4 and Section 5 that $\widetilde{\omega}_{\tau}$ may not satisfy the hard Lefschetz property, even the manifold has the same diffeomorphism type with the K\"{a}hler manifold $(N,\sigma,J)$.

Here is the final remark.
We can lift any symplectic deformation on $N$ to $\widetilde{M}(N,e,2\epsilon)_{-\epsilon, \epsilon}$ as follows.
\begin{lemma}\label{lemma : symplectic deformation}
    Let $N$ be a compact manifold, $\sigma$ and $\gamma$ be two symplectic forms on $N$ and $e \in H^2(N;\Z)$ be an integral class. Assume that $\epsilon > 0$ is chosen such that
    $\widetilde{\omega}_{\sigma}$ and $\widetilde{\omega}_{\gamma}$ are symplectic forms on $\widetilde{M}(N,e,2\epsilon)_{-\epsilon, \epsilon}$. If $\gamma$ is deformation equivalent to $\sigma$, then the induced symplectic form $\widetilde{\omega}_{\gamma}$ is also deformation equivalent to $\widetilde{\omega}_{\sigma}$.
\end{lemma}

\begin{proof}
    Let $\{ \sigma_t \}_{0 \leq t \leq 1}$ be a path of symplectic forms such that $\sigma_0 = \sigma$ and $\sigma_1 = \gamma$.
    Then $$\omega_{\sigma_t} = \pi^*\sigma_t + d(r\cdot \theta)$$ is a path of symplectic forms on $M(N,e,2\epsilon) = P \times (-2\epsilon, 2\epsilon)$ which connects $\omega_{\sigma}$ with $\omega_{\gamma}$. Then
    $$\omega_{\sigma_t} + \frac{1}{2}\sum_{i=1}^2 \frac{1}{\sqrt{-1}} dz_i \wedge d\bar{z_i}$$
    is a path of symplectic forms on $M(N,e,2\epsilon) \times \C^2$.
    Therefore $\widetilde{\omega}_{\sigma_t}$ is a path of reduced symplectic forms on $\mu^{-1}(-\epsilon,\epsilon) / T_2 = \widetilde{M}(N,e,2\epsilon)_{-\epsilon, \epsilon}$ which connects
    $\widetilde{\omega}_{\sigma}$ with $\widetilde{\omega}_{\gamma}$.

\end{proof}

\section{Equivariant cohomology theory for Hamiltonian circle actions}
    In this section, we briefly review the classical facts about equivariant cohomology theory for Hamiltonian circle actions which
    will be used in Section 4. Throughout this section, we assume that every coefficient of any cohomology theory is $\R$. Let $S^1$ be the unit
    circle group and let $M$ be an $S^1$-manifold. Then the equivariant cohomology ring $H^*_{S^1}(M)$ is defined by $$H^*_{S^1}(M) := H^*(M \times_{S^1} ES^1)$$
    where $ES^1$ is a contractible space on which $S^1$ acts freely. Let $BS^1 = ES^1 / S^1$ be the classfying space of $S^1$. Note that $H^*(BS^1; \R)$ is isomorphic to the polynomial ring $\R[u]$ where $u$ is a positive generator of degree two
    such that $\langle u, [\C P^1] \rangle = 1$ for $\C P^1 \subset \C P^2 \subset \cdots \subset \C P^\infty \cong BS^1.$

    Now, let $M^{S^1}$ be the fixed point set. Then the inclusion map $i : M^{S^1} \hookrightarrow M$ induces a ring homomorphism $$i^* :
    H^*_{S^1}(M) \rightarrow H^*_{S^1}(M^{S^1}) \cong \bigoplus_{F \subset M^{S^1}} H^*(F)\otimes H^*(BS^1)$$ and we call $i^*$ \textit{the restriction map to the fixed point set}.
    For any connected fixed component $F \in M^{S^1}$, the inclusion map $i_F : F \hookrightarrow M^{S^1}$ induces a natural projection
    $$i_F^* : H^*_{S^1}(M^{S^1}) \rightarrow H^*_{S^1}(F) \cong H^*(F)\otimes H^*(BS^1)$$
    and we denote by $\alpha|_F$ an image $i_F^* (i^*(\alpha))$ for each $\alpha \in H^*_{S^1}(M)$.

    Let $\alpha \in H^k_{S^1}(M)$ be a class of degree $k$. For each fixed component $F \subset M^{S^1}$, let $j$ be the smallest positive integer such that
                       $$\alpha|_F \in \bigoplus_{i=0}^j  H^*(F) \otimes H^j(BS^1).$$
    We call such a number $j$ \textbf{the $H^*(BS^1)$-degree} of $\alpha|_F$.
    \begin{remark}
        McDuff and Tolman \cite[page 8]{McT} called the $H^*(BS^1)$-degree of $\alpha|_F$ a \textbf{degree} of $\alpha|_F$. But the author thought that it might lead to confusion with the degree of $\alpha$ as a cohomology class. Hence in this paper, we use the term `\textbf{$H^*(BS^1)$-degree}' instead of the word `\textbf{degree}' to avoid confusion with the standard use of `\textbf{degree}' of a cohomology class.
    \end{remark}

    If $(M,\omega)$ is a symplectic manifold equipped with a Hamiltonian circle action, then the equivariant cohomology ring $H^*_{S^1}(M)$ has several remarkable properties as follows.

    \begin{theorem}\cite{Ki}\label{theorem : injectivity}
    Let $(M,\omega)$ be a closed symplectic manifold and $S^1$ acts on $(M,\omega)$ in a Hamiltonian fashion. Then the restriction map
    $$i^* : H^*_{S^1}(M) \rightarrow H^*_{S^1}(M^{S^1})$$ is injective.
    \end{theorem}
    Theorem \ref{theorem : injectivity} is called the \textit{Kirwan's injectivity theorem}, and it tells us that any class $\alpha \in H^*_{S^1}(M)$ is uniquely determined by its image of $i^*$.

    \begin{theorem}\cite{Ki}\label{theorem : formality}
        Let $(M,\omega)$ be a smooth compact symplectic manifold with a Hamiltonian circle action. Then $M$ is equivariantly formal, i.e., the Leray-Serre spectral sequence associated to the fibration $M \times_{S^1} ES^1 \rightarrow BS^1$ collapses at the $E_1$ stage.
    \end{theorem}

    Note that $M \times_{S^1} ES^1$ has an $M$-bundle structure over $BS^1$ so that $H^*_{S^1}(M)$ has an $H^*(BS^1)$-module structure. Here is another description of \textit{equivariant formality} of $H^*_{S^1}(M)$ as follows.

    \begin{theorem}\cite{Ki}\label{theorem : formality2}
        $M$ is equivariantly formal if and only if $H^*_{S^1}(M)$ is a free $H^*(BS^1)$-module. Also, $M$ is equivariantly formal if and only if for the inclusion of a fiber $f : M \hookrightarrow M \times_{S^1} ES^1$, the induced ring homomorphism $f^* : H^*_{S^1}(M) \rightarrow H^*(M)$ is surjective. The kernel of $f^*$ is given by $u \cdot H^*_{S^1}(M)$ where $\cdot$ means the scalar multiplication for the $H^*(BS^1)$-module structure on $H^*_{S^1}(M)$.
    \end{theorem}

    Now, let us focus on our situation. Assume that $(M,\omega)$ is a $2n$-dimensional smooth compact symplectic manifold equipped with a Hamiltonian circle action with a moment map $H : M \rightarrow \R$. It is well-known \cite{Au} that $H$ is a Morse-Bott function and for every critical point of $H$, a stable (unstable, respectively) submanifold is defined with respect to the gradient vector field of $H$.
    For each fixed component $F \subset M^{S^1}$, let $\nu_{F}$ be a normal bundle of $F$ in $M$. Then the \textit{negative normal bundle} $\nu^{-}_{F}$ of $F$ can be defined as
    a sub-bundle of $\nu_F$ whose fiber over $p \in F$ is a subspace of $T_p M$ tangent to the unstable submanifold of $M$ at $F$ for every $p \in F$. We denote by $e^{-}_F \in
    H^*_{S^1}(F)$ the equivariant Euler class of $\nu^{-}_F$.
    Since $H^*_{S^1}(M)$ is a free $H^*(BS^1)$-module by Theorem \ref{theorem : formality2}, every cohomology class $\alpha \in H^k_{S^1}(M)$ can be uniquely expressed by
    $$ \alpha = \alpha_k \otimes 1 + \alpha_{k-2} \otimes u + \cdots \in H^k_{S^1}(M) \cong H^k(M) \otimes H^0(BS^1) \oplus \cdots$$  and it satisfies $f^*(\alpha) = \alpha_k$ where $f^*$ is the restriction to the fiber $M$ described in Theorem \ref{theorem : formality2}.
    In \cite{McT}, McDuff and Tolman found a basis of $H^*_{S^1}(M)$ whose elements are easily understood in terms of their restriction images to the fixed point set as follows.

    \begin{theorem}\cite{McT}\label{theorem : canonical basis}
        Let $(M,\omega)$ be a closed symplectic manifold equipped with a Hamiltonian circle action with a
        moment map $H : M \rightarrow \R$. For each connected fixed component $F \subset M^{S^1}$, let $k_F$ be a Morse index of $F$
        with respect to $H$. For a given any cohomology class $Y \in H^i(F)$, there exists a
        unique class $\widetilde{Y} \in H_{S^1}^{i + k_F}(M)$
        such that
        \begin{enumerate}
            \item $\widetilde{Y}|_{F'} = 0$ for every $F' \in M^{S^1}$ with $H(F') < H(F)$,
            \item $\widetilde{Y}|_F = Y \cup e^{-}_F$ with $Y = Y \otimes 1 \in H^*_{S^1}(F)$, and
            \item the $H^*(BS^1)$-degree of $\widetilde{Y}|_{F'} \in H^*_{S^1}(F')$ is less than the index $k_{F'}$ of $F'$ for all fixed components $F' \neq F.$
        \end{enumerate}
        We call such a class $\widetilde{Y}$ \textbf{the canonical class with respect to} $Y$. If we fix a basis $S_F$ of an $\R$-vector space $H^*(F)$ for each fixed compoent $F \subset M^{S^1}$, then the set
        $$\frak{B} = \{ \widetilde{Y} | Y \in \cup_{F \subset M^{S^1}} S_F \}$$ is a basis of $H^*_{S^1}(M)$ as an $H^*(BS^1)$-module. Moreover, $f^*(\frak{B})$ is a basis of $H^*(M)$ as an $\R$-vector space.
    \end{theorem}

    Using Theorem \ref{theorem : canonical basis}, we can check the hard Lefschetz property of $(M^{2n},\omega)$ via the Atiyah-Bott-Berline-Vergne localization theorem as follows.
    Denote by $HR_k : H^k(M) \times H^k(M) \rightarrow \R$ the Hodge-Riemann form defined by
    $$ HR_k(\alpha, \beta) = \langle \alpha\beta[\omega]^{n-k}, [M] \rangle $$ where $\alpha, \beta \in H^k(M)$ and $[M] \in H_{2n}(M;\Z)$ is the fundamental homology class of $M$. Then $(M,\omega)$ satisfies the hard Lefschetz property if and only if $HR_k$ is non-singular for every $k=0,1 \cdots, n.$ For each fixed component $F \subset M^{S^1}$, fix an $\R$-basis $S_F \subset H^*(F)$ and let $\frak{B} = \{\widetilde{Y} | Y \in \cup_{F \subset M^{S^1}} S_F\}$ the set of canonical classes with respect to the elements in $\cup_{F \subset M^{S^1}} S_F$. We denote by $\frak{B}_k := \frak{B} \cap H^k_{S^1}(M)$ and let $f^*(\frak{B}_k) = \{Y_1^k, \cdots, Y_{b_k}^k\} \subset H^k(M)$ where $b_k = b_k(M)$ is the $k$-th Betti number of $M$.
    Since $f^*(\frak{B}_k)$ is a basis of $H^k(M)$ by Theorem \ref{theorem : canonical basis}, it is obvious that the number of elements in $\frak{B}_k$ is $b_k$ and the set $\{Y_1^k, \cdots, Y_{b_k}^k\}$ are linearly independent. Therefore, $HR_k$ is non-singular if and only if the following matrix
    \begin{equation}\label{equation : Hodge-Riemann bilinear form}
        HR_k(\frak{B}_k) := (HR_k(Y_i^k, Y_j^k))_{i,j}
    \end{equation}
    is non-singular.
    To compute each component $HR_k(Y_i^k, Y_j^k)$, we use the Atiyah-Bott-Berline-Vergne localization theorem for circle actions as follows.

    \begin{theorem}[Atiyah-Bott-Berline-Vergne localization theorem]\label{localization}
        Let $M$ be a compact manifold with a circle action. Let $\alpha \in H^*_{S^1}(M;\R)$. Then as an element of $\R(u)$, we have
    $$\int_M \alpha = \sum_{F \in M^{S^1}} \int_F \frac{\alpha|_F}{e_F}$$ where the sum is taken over all fixed points, and $e_F$ is the equivariant Euler class of the normal bundle to $F$.
    \end{theorem}

    The integral $\int_M$ is often called the \textit{integration along the fiber $M$}. Since the action is Hamiltonian, $M$ is equivariantly formal by Theorem \ref{theorem : formality} so that we have $H^*_{S^1}(M) \cong H^*(M) \otimes H^*(BS^1)$ as an $H^*(BS^1)$-module by Theorem \ref{theorem : formality2}. If we denote by $\widetilde{Y_i^k} \in H^k_{S^1}(M)$ the canonical class such that $f^*(\widetilde{Y_i^k}) = Y_i^k \in H^k(M)$, then $\widetilde{Y_i^k}$ can be written by
    $$ \widetilde{Y_i^k} = Y_i^k \otimes 1 + \cdots \in H^k_{S^1}(M)$$ and the operation $\int_M$ acts on the ordinary cohomology factor, i.e., $\int_M \widetilde{Y_i^k} = \langle Y_i^k , [M] \rangle$. Therefore, we have
    \begin{equation}\label{equation : equivariant Hodge-Riemann bilnear form}
    HR_k(Y_i^k, Y_j^k) = \int_M \widetilde{Y_i^k}\widetilde{Y_j^k}\widetilde{\omega}^{n-k}
    \end{equation} where $\widetilde{\omega}$ is any equivariant extension of $\omega$, i.e., $f^*([\widetilde{\omega}]) = [\omega]$. The right-hand side of the equation (\ref{equation : equivariant Hodge-Riemann bilnear form}) seems to be complicated, but we will see that we can easily compute the integration in (\ref{equation : equivariant Hodge-Riemann bilnear form}) in Section 4.

    Here is the final remark. For a given Hamiltonian $S^1$-manifold $(M,\omega)$ with a moment map $H : M \rightarrow \R$, we may always find an equivariant extension $\widetilde{\omega}$ of $\omega$ on $M \times_{S^1} ES^1$ as follows. For the product space $M \times ES^1$, consider a two form $\omega_H := \omega + d(H \cdot \theta)$, regarding $\omega$ as a pull-back
    of $\omega$ along the projection $M \times ES^1 \rightarrow M$ and $\theta$ as a pull-back of a connection 1-form on the principal $S^1$-bundle $ES^1 \rightarrow BS^1$ along the projection $M \times ES^1 \rightarrow ES^1$. Here, the connection form $\theta$ is nothing but a finite dimensional approximation of the connection form of the principal $S^1$-bundle $S^{2n-1} \rightarrow \C P^n$. (See \cite{Au} for the details.) It is not hard to show that $\omega_H$ is $S^1$-invariant and
    $i_X\omega_H = 0$, i.e., the fundamental vector field generated by the action (tangent to the fiber $S^1$ of $M \times ES^1 \rightarrow M \times_{S^1} ES^1$) is in the kernel of $\omega_H$. Hence we may push-forward $\omega_H$ to the Borel construction $M \times_{S^1} ES^1$ and denote by $\widetilde{\omega}_H$ the push-forward
    of $\omega_H$. Obviously, the restriction of $\widetilde{\omega}_H$ on each fiber $M$ is
    precisely $\omega$ and we call a class $[\widetilde{\omega}_H] \in H^2_{S^1}(M)$ an \textit{equivariant symplectic class with respect
    to $H$.} By definition of $\widetilde{\omega}_H$, we have the following proposition.

    \begin{proposition}\label{proposition : restriction of symplectic class}\cite{Au}
    Let $(M,\omega)$ be an Hamiltonian $S^1$-manifold with a moment map $H : M \rightarrow \R$. Let $[\widetilde{\omega}_H]$ be the equivariant symplectic class with respect to $H$. Then for any fixed component $F \subset M^{S^1}$, we have
    $$[\widetilde{\omega}_H]|_F = -H(F) \otimes u + \omega|_F \otimes 1 \in H^*_{S^1}(F) \cong H^*(F) \otimes H^*(BS^1).$$
    \end{proposition}

\section{Proof of Proposition \ref{proposition : if and only if} and Proposition \ref{proposition : off-topic}}

In this section, we prove Proposition \ref{proposition : if and only if} and Proposition \ref{proposition : off-topic}.
Throughout this section, we assume that $(M,\omega)$ is a six-dimensional compact symplectic manifold with a Hamiltonian circle action with only two fixed components, i.e., a simple Hamiltonian $S^1$-manifold (see Definition \ref{definition : simple Hamiltonian} or \cite{HH}).
We also assume that a symplectic form $\omega$ on $M$ is chosen such that a moment map $H$ maps $M$ onto $[0,1]$.
We use the following notation :
\begin{itemize}
    \item $Z_{\min} = H^{-1}(0)$ ($Z_{\max} = H^{-1}(1)$, respectively) : the fixed component which attains a minimum (maximum, respectively) of $H$.
    \item $\omega_0 = \omega|_{Z_{\min}}$ ($\omega_1 = \omega|_{Z_{\max}}$, respectively) : the restriction of $\omega$ to $Z_{\min}$ ($Z_{\max}$, respectively).
\end{itemize}
Without loss of generality, we may assume that $\dim Z_{\min} \leq \dim Z_{\max}$.
Since every fixed component is a symplectic submanifold of $(M,\omega)$, the possible dimension of $Z_{\max}$ is 0, 2, or 4. Recall that $H$ is a perfect Morse-Bott function \cite{Ki}, i.e.,
    $$P_t(M) = \sum_{Z \subset M^{S^1}} (\sum_k \dim H^k(Z)t^{k + \mathrm{ind}(Z)}).$$
where $P_t(M)$ is the Poincar\'{e} polynomial of $M$.
In our situation, the Poincar\'{e} polynomial is written as
\begin{equation}\label{equation : poincare}
     P_t(M) = \sum_k \dim H^k(Z_{\min}) t^k + \sum_k \dim H^k(Z_{\max}) t^{k + \mathrm{ind}(Z_{\max})}.
\end{equation}
Before we proceed further, we need to introduce the following remarkable result due to Li and Tolman which will be used in the next step.
\begin{theorem}\cite[Theorem 1]{LT}\label{theorem : Li and Tolman}
    Let the circle act in a Hamiltonian fashion on a compact $2n$-dimensional symplectic manifold $(M,\omega)$. Suppose that $M^{S^1}$ has exactly two components, $X$ and $Y$, and $M^{S^1}$ is \textit{minimal}, i.e., $$\dim X + \dim Y = \dim M - 2.$$ Then
    \begin{itemize}
        \item $H^*(X;\Z) \cong \Z[u] / u^{i+1}$ where $\dim X = 2i$, and
        \item $H^*(Y;\Z) \cong \Z[v] / v^{j+1}$ where $\dim Y = 2j$.
    \end{itemize}
    Consequently, we have $H^i(M ;\Z) = H^i(\C P^n;\Z)$ for every $i \geq 0$.
\end{theorem}

In fact, Theorem \ref{theorem : Li and Tolman} is a part of the original theorem 1 in \cite{LT}. Actually they classified all possible cohomology rings of simple Hamiltonian $S^1$-manifold in the case where the fixed point set is minimal. Since it is enough to use Theorem \ref{theorem : Li and Tolman} in our paper, we omit the rest part of the original theorem 1 in \cite{LT}.

\begin{lemma}\label{lemma : dimension less than four}
    If $\dim Z_{\max} \leq 2$, then $(M,\omega)$ satisfies the hard Lefschetz property.
\end{lemma}

\begin{proof}
    Assume that $\dim Z_{\max} = 0$ ($Z_{\max} = \mathrm{pt}$) so that $\mathrm{ind}(Z_{\max}) = 6$ where $``\mathrm{ind}"$ means a Morse index with respect to $H$.
    By our assumption, we have $\dim Z_{\min} \leq \dim Z_{\max}$ so that $\dim Z_{\min} = 0$ with $\mathrm{ind}(Z_{\min}) = 0$.
    Hence the Poincar\'{e} polynomial in (\ref{equation : poincare}) is given by $$P_t(M) = t^6 + 1$$ so that we have $H^2(M) = H^4(M)= 0$
    which contradicts that $[\omega] \neq 0$ in $H^2(M)$.

    Now, assume that $\dim Z_{\max} = 2$ so that $\mathrm{ind}(Z_{\max}) = 4$. If $\dim Z_{\min} = 0$,
    Then the Poincar\'{e} polynomial in (\ref{equation : poincare}) is given by
    $$P_t = t^6 + b_1(Z_{\max})t^5 + t^4 + 1$$
    which is impossible by Poincar\'{e} duality of $M$.
    Hence we have $\dim Z_{\min} = 2$ so that it satisfies $$ \dim Z_{\max} + \dim Z_{\min} = \dim M - 2, $$
    i.e., the fixed point set is minimal in the sense of Li and Tolman \cite{LT}. By Theorem \ref{theorem : Li and Tolman}, we have $H^i(M ; \Z) \cong H^i(\C P^3 ; \Z)$ for every $i \geq 0$.
    Consequently, the hard Lefschetz property of $\omega$ is automatically satisfied.
\end{proof}

\begin{lemma}\label{lemma : dimension zero and four}
    If $\dim Z_{\max} = 4$ and $\dim Z_{\min} = 0$, then $(M,\omega)$ satisfies the hard Lefschetz property.
\end{lemma}

\begin{proof}
    If $\dim Z_{\max} = 4$, then $\mathrm{ind}(Z_{\max}) = 2$ so that the Poincar\'{e} polynomial in (\ref{equation : poincare}) is given by
    $$ P_t(M) = t^6 + b_3(Z_{\max})t^5 + b_2(Z_{\max})t^4 + b_1(Z_{\max})t^3 + t^2 + 1. $$
    By Poincar\'{e} duality, we have $b_3(Z_{\max}) = b_1(Z_{\max}) = 0$ and $b_2(Z_{\max}) = 1$.
    In particular, we have $b_2(M) = 1$ so that $\omega$ satisfies the hard Lefschetz property.
\end{proof}

\begin{lemma}\label{lemma : computation of eauivariant Euler class}
    Let $\pi : E \rightarrow B$ be a complex line bundle such that $c_1(E) = e \in H^2(B)$. Also suppose that $S^1$ acts linearly on $E$ such that $\pi$ is $S^1$-equivariant with respect to the trivial $S^1$-action on $B$, i.e., $S^1$ acts on $E$ fiberwise. Then the equivariant first Chern class of $E$ is given by
    $$c_1^{S^1}(E) = e \otimes 1 + 1 \otimes \lambda u \in H^2_{S^1}(B) = H^*(B) \otimes H^*(BS^1)$$
    where $\lambda$ is the unique non-zero weight of tangential $S^1$-representation on the fixed component $B$.
\end{lemma}

\begin{proof}
    Note that the equivariant Chern class $c_1^{S^1}(E)$ is the first Chern class of the complex line bundle
    $$ \widetilde{\pi} : E \times_{S^1}ES^1 \rightarrow B \times_{S^1} ES^1. $$
    Since $B \times_{S^1} ES^1 \cong B \times BS^1$, we have
    $$ c_1^{S^1}(E) = \alpha \otimes 1 + 1 \otimes k u \in H^*(B) \otimes H^*(BS^1)$$
    for some $\alpha \in H^2(B)$ and $k \in \Z$.
    It is not hard to show that $\alpha$ is the first Chern class of $\widetilde{\pi}^{-1}(B \times \mathrm{pt})$, i.e., the first Chern class of
    $$ \widetilde{\pi}|_B : E \times_{S^1} S^1 \subset E \times_{S^1} ES^1 \rightarrow B \times \mathrm{pt} $$
    which is isomorphic to $E$. Therefore we have $\alpha = e$.

    Similarly, $\lambda u$ is the first Chern class of $\widetilde{\pi}^{-1}(\mathrm{pt} \times BS^1)$, i.e., the first Chern class of
    $$ \C (\cong \pi^{-1}(\mathrm{pt})) \times_{S^1} ES^1 \rightarrow \mathrm{pt} \times BS^1. $$
    Hence we have $k = \lambda$ where $\lambda$ is the weight of the $S^1$-representation on the fiber $\pi^{-1}(\mathrm{pt})$ of $\pi$.
\end{proof}

Now we are ready to prove Proposition \ref{proposition : off-topic}.
\begin{proof}[Proof of Proposition \ref{proposition : off-topic}]
    By Lemma \ref{lemma : dimension less than four} and Lemma \ref{lemma : dimension zero and four}, we need only prove the case where
    $\dim Z_{\min} = 2$ and $\dim Z_{\max} = 4$ with $\mathrm{ind}(Z_{\max}) = 2$.
    In this case, the Poincar\'{e} polynomial in (\ref{equation : poincare}) is given by
    $$ P_t(M) = t^6 + b_3(Z_{\max})t^5 + b_2(Z_{\max})t^4 + b_1(Z_{\max})t^3 + 2t^2 + b_1(Z_{\min})t + 1 $$
    so that we have $b_2(Z_{\max}) = 2$ and $b_3(Z_{\max}) = b_1(Z_{\min})$
    by Poincar\'{e} duality. In particular, it follows that $b_2(M) = 2$ and $b_1(M) = b_1(Z_{\min})$.

    To prove the hard Lefschetz property of $\omega$, it is enough to show that the Hodge-Riemann bilinear forms $HR_1$ and $HR_2$ are non-singular since $\dim M = 6$ (see (\ref{equation : Hodge-Riemann bilinear form}) and (\ref{equation : equivariant Hodge-Riemann bilnear form}) in Section 3). Recall that $f^* : H^*_{S^1}(M) \rightarrow H^*(M)$ is the restriction map to the fiber $M$ defined in Theorem \ref{theorem : formality2}.

    Firstly, we will show that the second Hodge-Riemann form $HR_2 : H^2(M) \times H^2(M) \rightarrow \R$ is non-singular. Let $\widetilde{\alpha} \in H^2_{S^1}(M)$ be the canonical class with respect to the fundamental cohomology class $\alpha \in H^2(Z_{\min})$ with $\langle \alpha, [Z_{\min}] \rangle = 1$, and let $\widetilde{\beta} \in H^2_{S^1}(M)$ be the canonical class with respect to $\beta = 1 \in H^0(Z_{\max})$.
    By Theorem \ref{theorem : canonical basis}, the set $\frak{B}_2 = \{f^*(\widetilde{\alpha}), f^*(\widetilde{\beta}) \}$ is an $\R$-basis of $H^2(M)$. Therefore, it is enough to show that the following matrix is non-singular.
    \begin{equation}\label{matrix : case 1}
    HR_2(\frak{B}_2) = \left(\begin{array}{cc}
        HR_2(f^*(\widetilde{\alpha}), f^*(\widetilde{\alpha})) & HR_2(f^*(\widetilde{\alpha}), f^*(\widetilde{\beta}))\\
        HR_2(f^*(\widetilde{\alpha}), f^*(\widetilde{\beta})) & HR_2(f^*(\widetilde{\beta}), f^*(\widetilde{\beta}))\\
    \end{array}\right)
    \end{equation}
    We will show that $HR_2(f^*(\widetilde{\alpha}), f^*(\widetilde{\alpha})) = 0$ and $HR_2(f^*(\widetilde{\alpha}), f^*(\widetilde{\beta})) \neq 0$.
    \begin{lemma}\label{lemma : in proof of Proposition 1.8}
        $HR_2(f^*(\widetilde{\alpha}), f^*(\widetilde{\alpha})) = 0.$
    \end{lemma}
    \begin{proof}
        Recall that for the moment map $H : M \rightarrow [0,1]$, there is an equivariant symplectic form $\widetilde{\omega}_H$ on
        $M \times_{S^1} ES^1$ such that
        \begin{displaymath}
            \begin{array}{lll}
                 [\omega_H]|_{Z_{\min}} & = & [\omega_0] \otimes 1 \in H^2_{S^1}(Z_{\min}) ~\mathrm{with} ~~\omega_0 = \omega|_{Z_{\min}}, ~\mathrm{and} \\

                 [\omega_H]|_{Z_{\max}} & = & [\omega_1] \otimes 1 - 1 \otimes u \in H^2_{S^1}(Z_{\max}) ~\mathrm{with} ~~\omega_1 = \omega|_{Z_{\max}} \\
            \end{array}
        \end{displaymath}
        by Proposition \ref{proposition : restriction of symplectic class}, where $u \in H^2(BS^1)$ is the positive generator of $H^*(BS^1)$ (see Section 3).
        By the localization theorem \ref{localization} and the equation (\ref{equation : equivariant Hodge-Riemann bilnear form}), we have
        \begin{displaymath}
            \begin{array}{lll}
                HR_2(f^*(\widetilde{\alpha}), f^*(\widetilde{\alpha})) & = & \displaystyle \int_M \widetilde{\alpha}^2 \cdot [\widetilde{\omega}_H] \\
                & = & \displaystyle \int_{Z_{\min}} \frac{(\widetilde{\alpha}|_{Z_{\min}})^2([\widetilde{\omega}_H]|_{Z_{\min}})}{e_{Z_{\min}}} + \displaystyle \int_{Z_{\max}} \frac{(\widetilde{\alpha}|_{Z_{\max}})^2 ([\widetilde{\omega}_H]|_{Z_{\max}} )}{e_{Z_{\max}}}. \\
            \end{array}
        \end{displaymath}
        Since $[\widetilde{\alpha}]|_{Z_{\min}} = \alpha \otimes 1 \in H^2_{S^1}(Z_{\min})$ by Theorem \ref{theorem : canonical basis}, it implies that
        $(\widetilde{\alpha}|_{Z_{\min}})^2 = \alpha^2 \otimes 1 = 0$ by our assumption that $\dim Z_{\min} = 2$.
        Therefore, we have
        \begin{displaymath}
            \begin{array}{lll}
                HR_2(f^*(\widetilde{\alpha}), f^*(\widetilde{\alpha})) & = & \displaystyle \int_{Z_{\max}} \frac{(\widetilde{\alpha}|_{Z_{\max}})^2 ([\widetilde{\omega}_H]|_{Z_{\max}} )}{e_{Z_{\max}}}. \\
            \end{array}
        \end{displaymath}
        Note that $\widetilde{\alpha}|_{Z_{\max}} = \alpha' \otimes 1 + 1 \otimes ku \in H^2_{S^1}(Z_{\max})$ for some $\alpha' \in H^2(Z_{\max})$.
        However, the $H^*(BS^1)$-degree of $\widetilde{\alpha}|_{Z_{\max}}$ is less than $\mathrm{ind}(Z_{\max}) = 2$ by Theorem \ref{theorem : canonical basis} so that $k$ must be zero, i.e., $\widetilde{\alpha}|_{Z_{\max}} = \alpha' \otimes 1 \in H^2(Z_{\max}) \otimes H^0(BS^1) \subset H^2_{S^1}(Z_{\max})$. Also, note that the action is semi-free since the action is assumed to be effective. Hence the weight of $S^1$-representation on the normal bundle over $Z_{\max}$ is $-1$ so that the equivariant Euler class $e_{Z_{\max}}$ is given by $$e_{Z_{\max}} = - u \otimes 1 - 1 \otimes e$$ by Lemma \ref{lemma : computation of eauivariant Euler class}, where $e$ is the Euler class of a principal $S^1$-bundle $H^{-1}(1-\epsilon) \rightarrow H^{-1}(1-\epsilon) / S^1 \cong Z_{\max}$ for a sufficiently small $\epsilon$.
        Similarly, we can easily see that $$e_{Z_{\min}} = 1 \otimes u^2 + c_1(\nu_{Z_{\min}}) \otimes u $$
        by Lemma \ref{lemma : computation of eauivariant Euler class} where $\nu_{Z_{\min}}$ is the normal bundle over $Z_{\min}$.

        Therefore,
        \begin{displaymath}
            \begin{array}{lll}
                HR_2(f^*(\widetilde{\alpha}), f^*(\widetilde{\alpha})) & = & \displaystyle \int_{Z_{\max}} \frac{(\widetilde{\alpha}|_{Z_{\max}})^2 ([\widetilde{\omega}_H]|_{Z_{\max}})}{e_{Z_{\max}}} \\[1em]
                & = & \displaystyle \int_{Z_{\max}} \frac{(\alpha' \otimes 1)^2 ([\omega_1] \otimes 1 - 1 \otimes u)}{-1 \otimes u - e \otimes 1} \\[1em]
                & = & \displaystyle\int_{Z_{\max}} \frac{-(\alpha'^2 \otimes u)(1 \otimes u^2 - e \otimes u + e^2 \otimes 1)}{(-1 \otimes u - e \otimes 1)(1 \otimes u^2 - e \otimes u + e^2 \otimes 1)} \\[1em]
                & = & \displaystyle \int_{Z_{\max}} - \frac{\alpha'^2 \otimes u^3}{ -1 \otimes u^3 } \\[1em]
                & = & \displaystyle \int_{Z_{\max}} \alpha'^2. \\[1em]
            \end{array}
        \end{displaymath}
        But by dimensional reason, we have
        $$ 0 = \int_M \widetilde{\alpha}^2 = \int_{Z_{\max}} \frac{(\alpha' \otimes 1)^2}{-1 \otimes u - e \otimes 1} = - \int_{Z_{\max}} \frac{\alpha'^2 \otimes u^2}{1 \otimes u^3} = -\frac{1}{u} \int_{Z_{\max}} \alpha'^2. $$
        Hence this completes the proof of Lemma \ref{lemma : in proof of Proposition 1.8}.
    \end{proof}

    Now, it remains to show that $HR_2(f^*(\widetilde{\alpha}), f^*(\widetilde{\beta})) \neq 0.$ Note that $\widetilde{\beta}|_{Z_{\min}} = 0$ and $\widetilde{\beta}|_{Z_{\max}} = e_{Z_{\max}}$ by Theorem \ref{theorem : canonical basis}. Hence we have
    \begin{displaymath}
    \begin{array}{lll}
        HR_2(f^*(\widetilde{\alpha}), f^*(\widetilde{\beta})) & = & \displaystyle \int_M \widetilde{\alpha}\widetilde{\beta}[\widetilde{\omega}_H] \\[1em]
        & = & \displaystyle \int_{Z_{\min}} \frac{(\widetilde{\alpha}|_{Z_{\min}})(\widetilde{\beta}|_{Z_{\min}})\cdot ([\omega_0] \otimes 1)}{e_{Z_{\min}}} + \displaystyle \int_{Z_{\max}} \frac{(\widetilde{\alpha}|_{Z_{\max}})(\widetilde{\beta}|_{Z_{\max}}) ([\omega_1] \otimes 1 - 1\otimes u)}{e_{Z_{\max}}}\\[1em]
        & = & \displaystyle \int_{Z_{\max}} \frac{(\widetilde{\alpha}|_{Z_{\max}})(\widetilde{\beta}|_{Z_{\max}}) ([\omega_1] \otimes 1 - 1 \otimes u)}{e_{Z_{\max}}} \\[1em]
        & = & \displaystyle \int_{Z_{\max}} \frac{(\widetilde{\alpha}|_{Z_{\max}})(e_{Z_{\max}}) ([\omega_1] \otimes 1 - 1 \otimes u)}{e_{Z_{\max}}}\\[1em]
        & = & \displaystyle \int_{Z_{\max}} (\widetilde{\alpha}|_{Z_{\max}}) ([\omega_1] \otimes 1).\\
    \end{array}
    \end{displaymath}
    On the other hand, by dimensional reason, we have
    \begin{displaymath}
    \begin{array}{lll}
        0 & = & \displaystyle \int_M \widetilde{\alpha}[\widetilde{\omega}_H] = \int_{Z_{\max}} \frac{(\widetilde{\alpha}|_{Z_{\max}})([\omega_1] \otimes 1 - 1 \otimes u)}{-1 \otimes u - e \otimes 1}\\[1em]
        & = & \displaystyle \int_{Z_{\max}} \frac{(\widetilde{\alpha}|_{Z_{\max}})([\omega_1] \otimes 1 - 1 \otimes u )( 1 \otimes u^2 - e \otimes u + e^2 \otimes 1)}{(-1 \otimes u - e \otimes 1)(1 \otimes u^2 -e \otimes u + e^2 \otimes 1)} \\[1em]
        & = & \displaystyle \int_{Z_{\max}} \frac{(\widetilde{\alpha}|_{Z_{\max}})[\omega_1] \otimes u^2 + \widetilde{\alpha}|_{Z_{\max}} \cdot e \otimes u^2 }{-1 \otimes u^3} \\
    \end{array}
    \end{displaymath}
    so that $$\displaystyle \int_{Z_{\max}} (\widetilde{\alpha}|_{Z_{\max}})[\omega_1] = -\int_{Z_{\max}} \widetilde{\alpha}|_{Z_{\max}} \cdot e. $$
    Also by dimensional reason, we have
    \begin{displaymath}
    \begin{array}{lll}
        0 & = & \displaystyle \int_M \widetilde{\alpha} = \int_{Z_{\min}} \frac{\widetilde{\alpha}|_{Z_{\min}}}{e_{Z_{\min}}} + \int_{Z_{\max}} \frac{\widetilde{\alpha}|_{Z_{\max}}}{-1 \otimes u - e \otimes 1}\\[1em]
        & = & \displaystyle \int_{Z_{\min}} \frac{\alpha}{1 \otimes u^2 + c_1(\nu_{Z_{\min}}) \otimes u} + \int_{Z_{\max}} \frac{\widetilde{\alpha}|_{Z_{\max}}}{-1 \otimes u - e \otimes 1}\\[1em]
        & = & \displaystyle \int_{Z_{\min}} \frac{\alpha \otimes u}{1 \otimes u^3} + \int_{Z_{\max}} \frac{(\widetilde{\alpha}|_{Z_{\max}})e \otimes u}{1 \otimes u^3}\\[1em]
        & = & \displaystyle \frac{1}{u^2} + \frac{1}{u^2}\int_{Z_{\max}} (\widetilde{\alpha}|_{Z_{\max}}) e.\\
    \end{array}
    \end{displaymath}
    Consequently, we have $HR_2(f^*(\widetilde{\alpha}), f^*(\widetilde{\beta})) = 1$ so that $HR_2$ is non-singular.

    To show that $HR_1 : H^1(M) \times H^1(M) \rightarrow \R$ is non-singular, let us consider a symplectic basis $\frak{A} = \{\alpha_1, \cdots, \alpha_{2g}\}$ of $H^1(Z_{\min})$ with respect to the intersection product on $Z_{\min}$ so that the associate matrix $Q(\frak{A})$ is given by
    \begin{equation}
        Q(\frak{A}) = \left(\begin{array}{cc}
        0 & I_g\\

        -I_g & 0\\
    \end{array}\right)
    \end{equation}
    where $g$ is a genus of $Z_{\min}.$
    We denote by $\widetilde{\alpha}_i \in H^1_{S^1}(M)$ the canonical class with respect to $\alpha_i$ for each $i=1, 2, \cdots, 2g$. Then
    the set $\{f^*(\widetilde{\alpha}_1), \cdots, f^*(\widetilde{\alpha}_{2g}) \}$ forms an $\R$-basis of $H^1(M)$ by Theorem \ref{theorem : canonical basis}.
    The following lemma induces the non-singularity of $HR_1$ so that it finishes the proof of Proposition \ref{proposition : off-topic}.
    \begin{lemma}
        For each $i$ and $j$, we have $HR_1(f^*(\widetilde{\alpha_i}), f^*(\widetilde{\alpha_j})) = \int_{Z_{\min}} \alpha_i \alpha_j$ so that the associate matrix of $HR_1$ with respect to $\{ f^*(\alpha_i) \}_i$ is $Q(\frak{A})$. In particular, $HR_1$ is non-singular.
    \end{lemma}
    \begin{proof}
        By dimensional reason, we have
    \begin{displaymath}
        \begin{array}{lll}
            0 & = & \displaystyle \int_M \widetilde{\alpha}_i\widetilde{\alpha}_j = \int_{Z_{\min}} \frac{\alpha_i \alpha_j}{e_{Z_{\min}}} + \int_{Z_{\max}} \frac{\widetilde{\alpha}_i|_{Z_{\max}} \cdot \widetilde{\alpha}_j|_{Z_{\max}}}{-1 \otimes u - e \otimes 1}\\[1em]
            & = & \displaystyle \int_{Z_{\min}} \frac{\alpha_i \alpha_j}{1 \otimes u^2 + c_1(Z_{\min}) \otimes u} + \int_{Z_{\max}} \frac{\widetilde{\alpha}_i|_{Z_{\max}} \cdot \widetilde{\alpha}_j|_{Z_{\max}}}{-1 \otimes u - e \otimes 1}\\[1em]
            & = & \displaystyle \frac{1}{u^2}\int_{Z_{\min}} \alpha_i \alpha_j + \frac{1}{u^2}\int_{Z_{\max}} \widetilde{\alpha}_i|_{Z_{\max}} \cdot \widetilde{\alpha}_j|_{Z_{\max}} \cdot e\\
        \end{array}
    \end{displaymath}
    and

    \begin{displaymath}
        \begin{array}{lll}
            0 & = & \displaystyle \int_M \widetilde{\alpha_i}\widetilde{\alpha_j}[\widetilde{\omega}_H]\\[1em]
              & = & \displaystyle \int_{Z_{\min}} \frac{\widetilde{\alpha_i}|_{Z_{\min}} \cdot \widetilde{\alpha_j}|_{Z_{\min}} \cdot ([\omega_0] \otimes 1)}{1 \otimes u^2 + c_1(\nu_{Z_{\min}}) \otimes u } + \int_{Z_{\max}} \frac{\widetilde{\alpha_i}|_{Z_{\max}} \cdot \widetilde{\alpha_j}|_{Z_{\max}} \cdot ([\omega_1] \otimes 1 - 1 \otimes u)}{-1 \otimes u - e \otimes 1}\\[1em]
              & = & \displaystyle \int_{Z_{\min}} \frac{\alpha_i \alpha_j [\omega_0] \otimes 1}{1 \otimes u^2 + c_1(\nu_{Z_{\min}}) \otimes u } + \int_{Z_{\max}} \frac{\widetilde{\alpha_i}|_{Z_{\max}} \cdot \widetilde{\alpha_j}|_{Z_{\max}} \cdot ([\omega_1] \otimes 1 - 1 \otimes u)}{-1 \otimes u - e \otimes 1}\\[1em]
              & = & \displaystyle 0 - \frac{1}{u} \int_{Z_{\max}} \widetilde{\alpha_i}|_{Z_{\max}} \cdot \widetilde{\alpha_j}|_{Z_{\max}}\cdot [\omega_1] - \frac{1}{u}\int_{Z_{\max}} \widetilde{\alpha_i}|_{Z_{\max}} \cdot \widetilde{\alpha_j}|_{Z_{\max}} \cdot e.\\
        \end{array}
    \end{displaymath}
    Combining two equations above, we have
    $$\int_{Z_{\min}} \alpha_i \alpha_j = \int_{Z_{\max}} \widetilde{\alpha_i}|_{Z_{\max}} \widetilde{\alpha_j}|_{Z_{\max}}[\omega_1].$$
    Therefore
    \begin{displaymath}
        \begin{array}{lll}
           HR_1(f^*(\widetilde{\alpha_i}), f^*(\widetilde{\alpha_j}))
           & = & \displaystyle \int_M \widetilde{\alpha_i}\widetilde{\alpha_j}[\widetilde{\omega}_H]^2 \\[1em]
           & = & \displaystyle \int_{Z_{\min}} \frac{\alpha_i\alpha_j[\omega_0] \otimes 1}{1 \otimes u^2 + c_1(\nu_{Z_{\min}}) \otimes u} \\[1em]
           &   & + \displaystyle\int_{Z_{\max}} \frac{\widetilde{\alpha_i}|_{Z_{\max}} \cdot \widetilde{\alpha_j}|_{Z_{\max}} \cdot ([\omega_1] \otimes 1 - 1 \otimes u)^2}{-u - e}\\[1em]
            & = & \displaystyle 0 + 2 \int_{Z_{\max}} \widetilde{\alpha_i}|_{Z_{\max}} \widetilde{\alpha_j}|_{Z_{\max}} [\omega_1] + \int_{Z_{\max}} \widetilde{\alpha_i}|_{Z_{\max}} \widetilde{\alpha_j}|_{Z_{\max}} e\\[1em]
            & = & \displaystyle \int_{Z_{\min}} \alpha_i \alpha_j.
        \end{array}
    \end{displaymath}

    \end{proof}
Therefore, we proved $HR_1$ and $HR_2$ are non-singular so that $\omega$ satisfies the hard Lefschetz property. This finishes the proof of Proposition \ref{proposition : off-topic}.
\end{proof}

Now, it remains to prove Proposition \ref{proposition : if and only if}. Let $(M,\omega)$ be a six-dimensional simple Hamiltonian $S^1$-manifold with a moment map $H : M \rightarrow [0,1]$ such that $Z_{\min}$ and $Z_{\max}$ are all four-dimensional.

\begin{remark}
    As we have seen in Section 1, we may identify $Z_{\min}$ with $Z_{\max}$ along the gradient flow of $H$. In the rest of this section, we always assume that $H^*(Z_{\min}) = H^*(Z_{\max})$ via this identification described in Section 1.
\end{remark}
Let us fix a basis $\frak{B} = \{\alpha_1, \cdots, \alpha_{b_2} \}$ of $H^2(Z_{\min})$ where $b_2 = b_2(Z_{\min})$. Let $\widetilde{\alpha}_i \in H^2_{S^1}(M)$ be the canonical class with respect to $\alpha_i \in H^2(Z_{\min})$ for each $i=1, \cdots, b_2,$ and let $\widetilde{\beta} \in H^2_{S^1}(M)$ be the canonical class with respect to $\beta = 1 \in H^0(Z_{\max})$. Then the set $\frak{B}_2 = \{f^*(\widetilde{\alpha}_1), \cdots, f^*(\widetilde{\alpha}_{b_2}), f^*(\widetilde{\beta}) \}$ forms a basis of $H^2(M)$ by Theorem \ref{theorem : canonical basis}.
Then the matrix $HR_2(\frak{B}_2)$ associated to the second Hodge-Riemann form $HR_2 : H^2(M) \times H^2(M) \rightarrow \R$ with respect to the basis $\frak{B}_2$ is of the following form
\begin{equation}
    HR_2(\frak{B}_2) = \left(\begin{array}{c|c}
        A(\alpha_i,\alpha_j) & B(\alpha_i,\beta)\\
        \hline
        C(\alpha_j,\beta) & D(\beta,\beta)\\
    \end{array}\right)
    \end{equation}
    where $A$ is a $(b_2 \times b_2)$-matrix with entries $A_{i,j} = HR_2(f^*(\widetilde{\alpha}_i), f^*(\widetilde{\alpha}_j))$, $B$ is
    a $(1 \times b_2)$-matrix with $B_{b_2 + 1, j} =  HR_2(f^*(\widetilde{\beta}), f^*(\widetilde{\alpha}_j))$, $C$ is a $b_2 \times 1$-matrix with
    $C_{i,b_2+1} = HR_2(f^*(\widetilde{\alpha}_i), f^*(\widetilde{\beta}))$, and $D = HR_2(f^*(\widetilde{\beta}), f^*(\widetilde{\beta}))$.

\begin{lemma}\label{lemma : Euler class}
    Let $e_{Z_{\min}}$ ($e_{Z_{\max}}$, respectively) be the equivariant Euler class of the normal bundle over $Z_{\min}$($Z_{\max}$, respectively). Then we have $e_{Z_{\min}} = 1 \otimes u + e \otimes 1$ ($e_{Z_{\max}} = -1 \otimes u - e \otimes 1$, respectively) where $e \in H^2(Z_{\min})$ is the first Chern class of the normal bundle over $Z_{\min}$ on $M$.
\end{lemma}

\begin{proof}
    Note that the first Chern class of the normal bundle over $Z_{\min}$ ($\Z_{\max}$, respectively) is $e$ ($-e$, respectively) by the assumption, and the action on $M$ is semifree because of the effectiveness of the action. Hence the weight at $Z_{\min}$ is $1$ and $-1$ at $Z_{\max}$. The rest of the proof is straightforward by Lemma \ref{lemma : computation of eauivariant Euler class}.
\end{proof}

\begin{lemma}\label{lemma : restriction of widetilde alpha}
    For each $\widetilde{\alpha}_i$, we have $\widetilde{\alpha}_i|_{Z_{\min}} = \alpha_i \otimes 1$ and $\widetilde{\alpha}_i|_{Z_{\max}} = \beta_i \otimes 1$ for some $\beta_i \in H^2(Z_{\max})$.
\end{lemma}

\begin{proof}
    Since $Z_{\min}$ is the minimum, the negative normal bundle of $Z_{\min}$ is of rank zero. Hence by Theorem \ref{theorem : canonical basis}, we have $\widetilde{\alpha}_i|_{Z_{\min}} = (\alpha_i \otimes 1) \cup e^-_{Z_{\min}} = \alpha_i \otimes 1$. For $Z_{\max}$, the restriction of $\widetilde{\alpha}_i$ to $Z_{\max}$ has of $H^*(BS^1)$-degree less than $\mathrm{ind}(Z_{\max}) = 2$  by Theorem \ref{theorem : canonical basis} again. Because $H^*(BS^1)$-degree is always even ($\deg u = 2$), so the $H^*(BS^1)$-degree of $\widetilde{\alpha}_i|_{Z_{\max}}$ is zero. Hence $\widetilde{\alpha}_i|_{Z_{\max}} = \beta_i \otimes 1$ for some $\beta_i \in H^2(Z_{\max})$.
\end{proof}

\begin{lemma}\label{lemma : same intersection}
    Let $\{\beta_1, \cdots, \beta_{b_2} \}$ be given in Lemma \ref{lemma : restriction of widetilde alpha}. For each $0 \leq i,j \leq b_2(Z_{\min})$, we have
        $$\int_{Z_{\min}} \alpha_i \alpha_j = \int_{Z_{\max}} \beta_i \beta_j.$$

\end{lemma}

\begin{proof}
    By applying the localization theorem \ref{localization} to $\widetilde{\alpha}_i \cdot \widetilde{\alpha}_j$, we have
    $$0 = \int_M \widetilde{\alpha}_i \cdot \widetilde{\alpha}_j = \int_{Z_{\min}} \frac{\widetilde{\alpha}_i|_{Z_{\min}} \cdot \widetilde{\alpha}_j|_{Z_{\min}}}{e_{Z_{\min}}} + \int_{Z_{\max}} \frac{\widetilde{\alpha}_i|_{Z_{\max}} \cdot \widetilde{\alpha}_j|_{Z_{\max}}}{e_{Z_{\max}}}. $$
    Applying Lemma \ref{lemma : Euler class} and Lemma \ref{lemma : restriction of widetilde alpha} to the equation above, we have
    $$ 0 = \int_{Z_{\min}} \frac{\alpha_i \alpha_j \otimes 1}{1 \otimes u + e \otimes 1} + \int_{Z_{\max}} \frac{\beta_i \beta_j \otimes 1}{-1 \otimes u - e \otimes 1}.$$
    Since $$\int_{Z_{\min}} \frac{(\alpha_i \alpha_j \otimes 1)(1 \otimes u^2 - e \otimes u + e^2 \otimes 1)}{(1 \otimes u + e \otimes 1)(1 \otimes u^2 - e \otimes u + e^2 \otimes 1)} = \frac{1}{u} \int_{Z_{\min}} \alpha_i \alpha_j$$ and $$\int_{Z_{\max}} \frac{(\beta_i \beta_j \otimes 1)(1 \otimes u^2 - e \otimes u + e^2 \otimes 1)}{(-1 \otimes u - e \otimes 1)(1 \otimes u^2 - e \otimes u + e^2 \otimes 1)} = -\frac{1}{u} \int_{Z_{\max}} \beta_i \beta_j$$ which finishes the proof.

%
\end{proof}

\begin{lemma}\label{lemma : alpha i and alpha j}
    Let $\{\beta_1, \cdots, \beta_{b_2} \}$ be given in Lemma \ref{lemma : restriction of widetilde alpha}.
    For each $0 \leq i,j \leq b_2(Z_{\min})$, we have $$HR_2(f^*(\widetilde{\alpha}_i), f^*(\widetilde{\alpha}_j)) = \langle \alpha_i\alpha_j, [Z_{\min}] \rangle =  \langle \beta_i \beta_j, [Z_{\max}] \rangle.$$
\end{lemma}

\begin{proof}
    By applying the localization theorem \ref{localization} to $$HR_2(f^*(\widetilde{\alpha}_i), f^*(\widetilde{\alpha}_j)) = \int_M \widetilde{\alpha}_i \widetilde{\alpha}_j [\widetilde{\omega}_H],$$ we have
    $$\int_M \widetilde{\alpha}_i \widetilde{\alpha}_j [\widetilde{\omega}_H] = \int_{Z_{\min}} \frac{\widetilde{\alpha}_i|_{Z_{\min}} \cdot \widetilde{\alpha}_j|_{Z_{\min}} \cdot [\widetilde{\omega}_H]|_{Z_{\min}}}{e_{Z_{\min}}} + \int_{Z_{\max}} \frac{\widetilde{\alpha}_i|_{Z_{\max}} \cdot \widetilde{\alpha}_j|_{Z_{\max}} \cdot [\widetilde{\omega}_H]|_{Z_{\max}}}{e_{Z_{\max}}}.
    $$
    Note that $$\int_{Z_{\min}} \frac{\widetilde{\alpha}_i|_{Z_{\min}} \cdot \widetilde{\alpha}_j|_{Z_{\min}} \cdot [\widetilde{\omega}_H]|_{Z_{\min}}}{e_{Z_{\min}}} = \int_{Z_{\min}} \frac{\alpha_i \alpha_j [\omega_0] \otimes 1}{1 \otimes u + e \otimes 1} = 0$$
    and
    $$\int_{Z_{\max}} \frac{\widetilde{\alpha}_i|_{Z_{\max}} \cdot \widetilde{\alpha}_j|_{Z_{\max}} \cdot [\widetilde{\omega}_H]|_{Z_{\max}}}{e_{Z_{\max}}} = \int_{Z_{\max}} \frac{(\beta_i\beta_j \otimes 1) \cdot ([\omega_1] \otimes 1 - 1 \otimes u)}{-1 \otimes u - e \otimes 1}$$
    by Lemma \ref{lemma : restriction of widetilde alpha} and Proposition \ref{proposition : restriction of symplectic class}. Therefore,
    \begin{displaymath}
        \begin{array}{lll}
            \displaystyle \int_M \widetilde{\alpha}_i \widetilde{\alpha}_j [\widetilde{\omega}_H] & = & \displaystyle \int_{Z_{\max}} \frac{(\beta_i \beta_j \otimes 1)\cdot ([\omega_1] \otimes 1 - 1 \otimes u)}{-1 \otimes u - e \otimes 1} \\[1em]
             & = & \displaystyle \int_{Z_{\max}} \frac{(\beta_i \beta_j \otimes 1) \cdot (-1 \otimes u)}{-1 \otimes u - e \otimes 1} = \int_{Z_{\max}} \beta_i \beta_j.
        \end{array}
    \end{displaymath}
    The second equality is straightforward by Lemma \ref{lemma : same intersection}.
\end{proof}

\begin{lemma}\label{lemma : alpha and beta}
    Let $\{\beta_1, \cdots, \beta_{b_2} \}$ be given in Lemma \ref{lemma : restriction of widetilde alpha}.
    For each $i$, we have $HR_2(f^*(\widetilde{\alpha}_i), f^*(\widetilde{\beta})) = \langle \beta_i [\omega_1], [Z_{\max}] \rangle.$
\end{lemma}

\begin{proof}
    Apply the localization theorem \ref{localization} again to
    $$HR_2(f^*(\widetilde{\alpha}_i), f^*(\widetilde{\beta})) = \int_M \widetilde{\alpha}_i \widetilde{\beta} [\widetilde{\omega}_H].$$
    Then
    $$
    \int_M \widetilde{\alpha}_i \widetilde{\beta} [\widetilde{\omega}_H] = \int_{Z_{\min}} \frac{\widetilde{\alpha}_i|_{Z_{\min}} \cdot \widetilde{\beta}|_{Z_{\min}} \cdot [\widetilde{\omega}_H]|_{Z_{\min}}}{e_{Z_{\min}}} + \int_{Z_{\max}} \frac{\widetilde{\alpha}_i|_{Z_{\max}} \cdot \widetilde{\beta}|_{Z_{\max}} \cdot [\widetilde{\omega}_H]|_{Z_{\max}}}{e_{Z_{\max}}}.
    $$
    By Theorem \ref{theorem : canonical basis}, we have $\widetilde{\beta}|_{Z_{\min}} = 0$ and $\widetilde{\beta}|_{Z_{\max}} = e_{Z_{\max}}$ so that
    $$\int_M \widetilde{\alpha}_i \widetilde{\beta} [\widetilde{\omega}_H] = \int_{Z_{\max}} (\beta_i \otimes 1) \cdot ([\omega_1] \otimes 1 - 1 \otimes u) = \int_{Z_{\max}} \beta_i [\omega_1].$$
\end{proof}

\begin{lemma}\label{lemma : beta beta}
    $HR_2(f^*(\widetilde{\beta}), f^*(\widetilde{\beta})) = \int_{Z_{\max}} [\omega_1]^2$ if and only if $[\omega_0]\cdot [\omega_1] = 0$.
\end{lemma}

\begin{proof}

    By the localization theorem \ref{localization} again, we have
    \begin{displaymath}
        \begin{array}{lll}
            HR_2(f^*(\widetilde{\beta}), f^*(\widetilde{\beta})) & = & \displaystyle \int_M \widetilde{\beta}^2 [\widetilde{\omega}_H]\\[1em]
             & = & \displaystyle \int_{Z_{\min}} \frac{(\widetilde{\beta}|_{Z_{\min}})^2 \cdot [\widetilde{\omega}_H]|_{Z_{\min}}}{e_{Z_{\min}}} + \int_{Z_{\max}} \frac{(\widetilde{\beta}|_{Z_{\max}})^2 \cdot [\widetilde{\omega}_H]|_{Z_{\max}}}{e_{Z_{\max}}}.\\
        \end{array}
    \end{displaymath}
    Since $\widetilde{\beta}|_{Z_{\min}} = 0$ and $\widetilde{\beta}|_{Z_{\max}} = e_{Z_{\max}}$, we have
    \begin{displaymath}
        \begin{array}{lll}
            \displaystyle \int_M \widetilde{\beta}^2 [\widetilde{\omega}_H] & = & \displaystyle \int_{Z_{\max}} e_{Z_{\max}} \cdot ([\omega_1] \otimes 1 - 1 \otimes u) \\[1em]
            & = & \displaystyle \int_{Z_{\max}} (-e \otimes 1 - 1 \otimes u) \cdot ([\omega_1] \otimes 1 - 1 \otimes u) = \int_{Z_{\max}} -e \cdot [\omega_1].
        \end{array}
    \end{displaymath}
    Note that $[\omega_1] = [\omega_0] - e$ by the Duistermaat-Heckman theorem \cite[Theorem VI.2.3]{Au}. Hence we have $$[\omega_0] \cdot [\omega_1] = [\omega_1]^2 + e \cdot [\omega_1].$$
    Hence $[\omega_1]^2 = -e \cdot [\omega_1]$ if and only if $[\omega_0] \cdot [\omega_1] = 0.$
    Therefore $[\omega_0] \cdot [\omega_1] = 0$ if and only if
    $$\int_M \widetilde{\beta}^2 [\widetilde{\omega}_H] = \int_{Z_{\max}} -e \cdot [\omega_1] = \int_{Z_{\max}} [\omega_1]^2.$$

\end{proof}

Now, we are ready to prove Proposition \ref{proposition : if and only if}.
\begin{proof}[Proof of Proposition \ref{proposition : if and only if}]
    Firstly, we prove that the second Hodge-Riemann form $HR_2 : H^2(M) \times H^2(M) \rightarrow \R$ is non-singular if and only if $[\omega_0] \cdot ~[\omega_1] \neq 0$. By Lemma \ref{lemma : restriction of widetilde alpha}, we can define a map $\phi$ such that
    \begin{displaymath}
        \begin{array}{lccccl}
             \phi & : & H^2(Z_{\min}) & \rightarrow & H^2(Z_{\max}) & \\
                  &   &  \alpha & \mapsto & \phi(\alpha), &  \widetilde{\alpha}|_{Z_{\max}} = \phi(\alpha) \otimes 1. \\
        \end{array}
    \end{displaymath}
    It is obvious that $\phi$ is $\R$-linear by Theorem \ref{theorem : canonical basis}, i.e., $\widetilde{\gamma_1} + \widetilde{\gamma_2}$ is the canonical class with respect to $\gamma_1 + \gamma_2$ for every $\gamma_1, \gamma_2 \in H^2(Z_{\min}).$ Furthermore, the map $\phi$ preserves the intersection product by Lemma \ref{lemma : same intersection}. Hence $\phi$ is an $\R$-isomorphism.

    Now, let $\{\alpha_1, \cdots, \alpha_{b_2} \}$ be an orthogonal basis of $H^2(Z_{\min})$ with respect to the intersection product on $H^*(Z_{\min})$ such that $\phi(\alpha_{b_2}) = [\omega_1] \in H^2(Z_{\max})$. Such $\alpha_{b_2}$ exists since $\phi$ is an isomorphism.
    Then the set $\{\beta_1, \cdots, \beta_{b_2} = [\omega_1] \}$ is also an orthogonal basis of $H^2(Z_{\max})$ with respect to the intersection product on $H^*(Z_{\max})$, where $\beta_i$ defined in Lemma \ref{lemma : restriction of widetilde alpha} is nothing but $\phi(\alpha_i)$ for every $i = 1, \cdots, b_2$.
    Let $\widetilde{\alpha_i} \in H^2_{S^1}(M)$ be the canonical class with respect to $\alpha_i$ for each $i$ and let $\widetilde{\beta} \in H^2_{S^1}(M)$ be the canonical class with respect to $1 \in H^0(Z_{\max})$ respectively. Then $\frak{B}_2 = \{f^*(\widetilde{\alpha}_1), \cdots, f^*(\widetilde{\alpha}_{b_2}), f^*(\widetilde{\beta}) \}$ is a basis of $H^2(M)$ by Theorem \ref{theorem : canonical basis}. Using Lemma \ref{lemma : alpha i and alpha j} and Lemma \ref{lemma : alpha and beta}, the matrix associated to $HR_2$ with respect to the basis $\frak{B} = \{f^*(\widetilde{\alpha}_1), \cdots, f^*(\widetilde{\alpha}_{b_2}), f^*(\widetilde{\beta}) \}$ is of the following form

    \begin{equation}
    HR_2(\frak{B}) = \left(\begin{array}{cccc|c}
        \int_{Z_{\max}} \beta_1^2 & 0 & \cdots & 0 & 0\\

        0 & \int_{Z_{\max}} \beta_2^2 & 0 & \cdots & 0\\

        \vdots & \vdots & \vdots & \vdots & \vdots\\

        0 & \cdots & 0 & \int_{Z_{\max}} \beta_{b_2}^2 & \int_{Z_{\max}} \beta_{b_2}^2 \\
        \hline
        0 & \cdots & 0 & \int_{Z_{\max}} \beta_{b_2}^2 & * \\

    \end{array}\right)
    \end{equation}
    Also, Lemma \ref{lemma : beta beta} implies that $[\omega_0] \cdot [\omega_1] = 0$ if and only if
    \begin{equation}
    HR_2(\frak{B}) = \left(\begin{array}{cccc|c}
        \int_{Z_{\max}} \beta_1^2 & 0 & \cdots & 0 & 0\\

        0 & \int_{Z_{\max}} \beta_2^2 & 0 & \cdots & 0\\

        \vdots & \vdots & \vdots & \vdots & \vdots\\

        0 & \cdots & 0 & \int_{Z_{\max}} \beta_{b_2}^2 & \int_{Z_{\max}} \beta_{b_2}^2 \\
        \hline
        0 & \cdots & 0 & \int_{Z_{\max}} \beta_{b_2}^2 & \int_{Z_{\max}} \beta_{b_2}^2 \\

    \end{array}\right)
    \end{equation}
    Hence $[\omega_0] \cdot [\omega_1] \neq 0$ if and only if the associated matrix $HR_2(\frak{B})$ is non-singular.

    Secondly, we prove that the first Hodge-Riemann form $HR_1 : H^1(M) \times H^1(M) \rightarrow \R$ is non-singular if and only if $(Z_{\max}, \omega_0 + \omega_1)$ satisfies the hard Lefschetz property.
    Note that for each element $\alpha \in H^1(Z_{\min})$, the restriction of the corresponding canonical class $\widetilde{\alpha}|_{Z_{\max}}$
    has $H^*(BS^1)$-degree less than one, i.e., $\widetilde{\alpha}|_{Z_{\max}} = \alpha' \otimes 1 \in H^1_{S^1}(Z_{\max})$ for some $\alpha' \in H^1(Z_{\max})$. Hence we can define a map
    \begin{displaymath}
        \begin{array}{lccccl}
             \psi & : & H^1(Z_{\min}) & \rightarrow & H^1(Z_{\max}) & \\
                  &   &  \alpha & \mapsto & \psi(\alpha), &  \widetilde{\alpha}|_{Z_{\max}} = \psi(\alpha) \otimes 1. \\
        \end{array}
    \end{displaymath}

    \begin{lemma}\label{lemma : psi H1 is isomorphism}
        The map $\psi : H^1(Z_{\min}) \rightarrow H^1(Z_{\max})$ is an isomorphism.
    \end{lemma}
    \begin{proof}
        Let $\alpha \in H^1(Z_{\min})$ be any non-zero class and let $\beta \in H^3(Z_{\min})$ such that $\alpha \cdot \beta \in H^4(Z_{\min})$ is non-zero. Such $\beta$ always exists because the intersection pairing $H^1(Z_{\min}) \times H^3(Z_{\min}) \rightarrow \R$ is non-singular. Let $\widetilde{\beta} \in H^3_{S^1}(M)$ be the canonical class with respect to $\beta$. Applying the localization theorem \ref{localization} to $\widetilde{\alpha}\widetilde{\beta} \in H^4_{S^1}(M)$, we have
        \begin{displaymath}
            \begin{array}{lll}
           0 & = & \displaystyle \int_M \widetilde{\alpha}\widetilde{\beta} \\[1em]
             & = & \displaystyle \int_{Z_{\min}} \frac{\widetilde{\alpha}|_{Z_{\min}} \widetilde{\beta}|_{Z_{\min}}}{e_{Z_{\min}}} - \int_{Z_{\max}} \frac{\widetilde{\alpha}|_{Z_{\max}} \widetilde{\beta}|_{Z_{\max}}}{e_{Z_{\max}}} \\[1em]
             & = & \displaystyle \int_{Z_{\min}} \frac{\alpha\beta}{e \otimes 1 + 1 \otimes u} - \int_{Z_{\max}} \frac{(\psi(\alpha) \otimes 1) \cdot \widetilde{\beta}|_{Z_{\max}}}{e_{Z_{\max}}} \\[1em]
             & = & \displaystyle \frac{1}{u} \int_{Z_{\min}} \alpha\beta - \int_{Z_{\max}} \frac{(\psi(\alpha) \otimes 1) \cdot \widetilde{\beta}|_{Z_{\max}}}{e_{Z_{\max}}} \\[1em]
            \end{array}
        \end{displaymath}
        Hence $\psi(\alpha) = \widetilde{\alpha}|_{Z_{\max}}$ can never be zero.
    \end{proof}

    For any $\alpha, \beta \in H^1(Z_{\min})$, let $\widetilde{\alpha}, \widetilde{\beta} \in H^1_{S^1}(M)$ be the canonical classes of $\alpha$ and $\beta$ respectively. Then
    \begin{displaymath}
        \begin{array}{lll}
            HR_1(f^*(\widetilde{\alpha}), f^*(\widetilde{\beta})) & = & \int_M \widetilde{\alpha}\widetilde{\beta}[\widetilde{\omega_H}]^2\\[1em]
            & = & \displaystyle \int_{Z_{\min}} \frac{\alpha \beta [\omega_0]^2 \otimes 1}{e_{Z_{\min}}} + \int_{Z_{\max}} \frac{\widetilde{\alpha}|_{Z_{\max}} \cdot \widetilde{\beta}|_{Z_{\max}} \cdot ([\omega_1] \otimes 1 - 1 \otimes u)^2}{e_{Z_{\max}}}.
        \end{array}
    \end{displaymath}
    Since $\alpha \beta [\omega_0]^2 \in H^6(Z_{\min})$, the first summand must be zero. On the other hand,

    \begin{displaymath}
        \begin{array}{lll}
            \displaystyle \int_{Z_{\max}} \frac{\widetilde{\alpha}|_{Z_{\max}} \cdot \widetilde{\beta}|_{Z_{\max}} \cdot ([\omega_1] \otimes 1 - 1 \otimes u)^2}{e_{Z_{\max}}} & = & \displaystyle \int_{Z_{\max}} \frac{\widetilde{\alpha}|_{Z_{\max}} \cdot \widetilde{\beta}|_{Z_{\max}} \cdot ([\omega_1] \otimes 1 - 1 \otimes u)^2}{- e \otimes 1 - 1 \otimes u}\\[1em]
            & = & \displaystyle \int_{Z_{\max}} \frac{\widetilde{\alpha}|_{Z_{\max}} \cdot \widetilde{\beta}|_{Z_{\max}} \cdot ((2[\omega_1] + e) \otimes u^3)}{u^3} \\[1em]
            & = & \displaystyle \int_{Z_{\max}} \frac{\widetilde{\alpha}|_{Z_{\max}} \cdot \widetilde{\beta}|_{Z_{\max}} \cdot (([\omega_1] + ([\omega_1]+ e)) \otimes u^3)}{u^3} \\[1em]
            & = & \displaystyle \int_{Z_{\max}} \psi(\alpha)\psi(\beta)([\omega_1] + [\omega_0]). \\
        \end{array}
    \end{displaymath}
    The last equality comes from the Duistermaat-Heckman theorem which states that $[\omega_1] = [\omega_0] - e$.
    Therefore, we have
    \begin{equation}\label{equation : HLP for H1}
    HR_1(f^*(\widetilde{\alpha}), f^*(\widetilde{\beta})) = \langle \psi(\alpha)\psi(\beta)([\omega_1] + [\omega_0]), [Z_{\max}] \rangle
    \end{equation}
    Note that the right hand side of the equation (\ref{equation : HLP for H1}) is the first Hodge-Riemann form for $\omega_1 + \omega_0$.
    Since $\psi$ is an isomorphism by Lemma \ref{lemma : psi H1 is isomorphism}, we can conclude that $HR_1$ is non-singular if and only if $(Z_{\max}, \omega_1 + \omega_0)$ satisfies the hard Lefschetz property.

\end{proof}

\begin{remark}
    It is not clear that $\omega_0 + \omega_1$ is symplectic, but the hard Lefschetz property is a cohomological condition.
    In fact, the cohomology class of the reduced symplectic form $\omega_{\frac{1}{2}}$ on $H^{-1}(\frac{1}{2}) / S^1$ is $\frac{1}{2}([\omega_0] + [\omega_1])$.
    Therefore, we can conclude that $HR_1$ is non-singular if and only if the reduced symplectic form $\omega_{\frac{1}{2}}$ satisfies the hard Lefschetz property.
\end{remark}

\section{Proof of Theorem \ref{theorem : main} and Examples}

 Recall that for any compact symplectic four manifold $(N,\sigma)$ with an integral cohomology class $e \in H^2(N;\Z)$, we constructed a six-dimensional compact simple Hamiltonian $S^1$-manifold $(\widetilde{M}(N,e,2\epsilon)_{-\epsilon, \epsilon}, \widetilde{\omega}_{\sigma})$ with a moment map
 $$H : \widetilde{M}(N,e,2\epsilon)_{-\epsilon, \epsilon} \rightarrow [-\epsilon, \epsilon] $$ in Section 2.

 \begin{proposition}\cite[Proposition 5.8~(i), p.156]{McS}\label{proposition : Section 5, construct simple with family of symplectic form}
    Let $I \subset \R$ be an interval and let $(N,\sigma)$ be a four dimensional compact symplectic manifold and $e \in H^2(N;\Z)$ be an integral cohomology class of $N$. Let $\{ \sigma_t \}_I$ be a one-parameterized family of symplectic forms on $N$ such that $[\sigma_t] - [\sigma_s] = (s-t)e$ for all $s,t \in I$. Let $P$ be a pricipal bundle whose first Chern class is $e$. Then there exists an $S^1$-invariant symplectic form $\omega$ on the manifold $P \times [-1,1]$ with a moment map equal to the projection $P \times [-1,1] \rightarrow [-1,1]$ and each reduced symplectic form is $\sigma_t$ for all $t \in [-1,1]$.
 \end{proposition}

 Proposition \ref{proposition : Section 5, construct simple with family of symplectic form} implies that if there is a family of symplectic forms
 $\{ \sigma_t \}_{\-1 \leq t \leq 1}$ on $N$ with $[\sigma_t] - [\sigma_s] = (s-t)e$ for all $s,t \in [-1, 1]$, then we may choose $\epsilon = 1$ such that our manifold $(\widetilde{M}(N,e,1+\delta)_{-1,1}, \widetilde{\omega}_{\sigma})$ is symplectic for sufficiently small $\delta > 0$. Therefore, if we find
 \begin{itemize}
    \item a smooth compact four manifold $N$,
    \item an integral class $e \in H^2(N;\Z)$, and
    \item a family of symplectic forms $\{ \sigma_t \}_{\-1 \leq t \leq 1}$ on $N$ satisfying $[\sigma_t] - [\sigma_s] = (s-t)e$ for all $s,t \in [-1, 1]$ such that $[\sigma_{-1}]\cdot [\sigma_1] = 0$,
 \end{itemize}
 then the corresponding manifold $(\widetilde{M}(N,e,1+\delta)_{-1,1}, \widetilde{\omega}_{\sigma})$ would not satisfy the hard Lefschetz property.
 Before proving Theorem \ref{theorem : main}, we give simple examples as follows.

 \begin{example}[violating the first condition of Proposition \ref{proposition : if and only if}]
    Let $(N,\sigma)$ be any smooth compact symplectic four manifold which does not satisfy the hard Lefschetz property. Let $e \in H^2(N)$ be any integral cohomology class. Since the non-singularity of symplectic structure is an open condition, we can find a sufficiently large integer $k$ such that $\{ \sigma + t\cdot \gamma \}_{\{-\frac{1}{k} \leq t \leq \frac{1}{k} \}}$ is a family of symplectic forms on $N$, where $\gamma$ is any fixed closed two form which represents the class $e$. Then $\{ \sigma_t := k\cdot \sigma + t \cdot \gamma \}_{-1 \leq t \leq 1 }$ is a family of symplectic forms with $[\sigma_{-1}] + [\sigma_1] = 2k[\sigma]$. Hence the corresponding manifold $(\widetilde{M}(N,e,1+\delta)_{-1,1}, \widetilde{\omega}_{k\sigma})$ does not satisfies the hard Lefschetz property by Proposition \ref{proposition : if and only if}.
 \end{example}

 Suppose that $N$ is a smooth compact four manifold with two symplectic forms $\sigma_{-1}$ and $\sigma_1$ with $\sigma_{-1} \wedge \sigma_1 \equiv 0$ on $N$. Assume that $\sigma_{-1}$ and $\sigma_1$ give a same orientation on $N$. Then we can easily show that $\sigma_t := \frac{1-t}{2}\sigma_{-1} +
 \frac{1+t}{2}\sigma_1$ is a symplectic form on $N$ for every $t \in [-1,1].$ Hence if $e = [\sigma_{-1}] - [\sigma_1] \in H^2(N;Z)$ is an integral class and
 $\sigma_{-1} \wedge \sigma_1 \equiv 0$ on $N$, then the corresponding manifold $(\widetilde{M}(N,e,1+\delta)_{-1,1}, \widetilde{\omega}_{\sigma})$ with $\sigma = \sigma_0$ violates second condition of Proposition \ref{proposition : if and only if}.

\begin{example}[violating the second condition of Proposition \ref{proposition : if and only if}]
    Let $N = T^4 \cong S^1 \times S^1 \times S^1 \times S^1$ with a component-wise coordinate system $$(x_1,x_2,x_3,x_4) \in S^1 \times S^1 \times S^1 \times S^1$$ such that $\int_{S^1} dx_i = 1$ for all $i$. Let
    \begin{displaymath}
    \begin{array}{lll}
        \sigma_{-1} & = & 2(dx_1 \wedge dx_2 + dx_3 \wedge dx_4), ~\mathrm{and}\\
        \sigma_1 & = & 2(dx_1 \wedge dx_4 + dx_2 \wedge dx_3). \\
    \end{array}
    \end{displaymath}
Then we have $$\sigma_{-1}^2 = \sigma_1^2 = 8 (dx_1 \wedge dx_2 \wedge dx_3 \wedge dx_4)$$ so that $\sigma_{-1}$ and $\sigma_1$ give the same orientation on $N$. Also, $[\sigma_{-1}]$ and $[\sigma_1]$ are integral classes in $H^2(T^4;\Z) \subset H^2(T^4;\R)$ and $\sigma_{-1} \wedge \sigma_1 \equiv 0$ on $T^4$. Hence the corresponding simple Hamiltonian $S^1$-manifold $(\widetilde{M}(N,e,1+\delta)_{-1,1}, \widetilde{\omega}_{\sigma_0})$ with $e = dx_1 \wedge dx_2 + dx_3 \wedge dx_4 - dx_1 \wedge dx_4 - dx_2 \wedge dx_3$ does not satisfy the hard Lefschetz property by Proposition \ref{proposition : if and only if}.
\end{example}

Before giving a proof of the main theorem \ref{theorem : main}, we need to recall holomorphic symplectic manifolds which will be used in the proof.
Let $(N,\sigma,J)$ be a compact K\"{a}hler manifold.
A \textit{holomorphic symplectic form} on $N$ is a closed and non-degenerate holomorphic two form $\rho \in \Omega^{2,0}(N)$.
If $N$ admits a holomorphic symplectic form $\rho$, then $\rho^n$ is nowhere vanishing so that it gives a nowhere zero section of the canonical line bundle $K_N = \wedge^n T^*N$, which means that $K_N$ is a trivial bundle.
Conversely, if $N$ is a compact K\"{a}hler manifold such that $K_N$ is trivial, then there exists a Ricci-flat metric on $N$ so that the structure group $U(n)$ can be reduced to $SU(n)$ (it was conjectured by E. Calabi and proved by S. T. Yau in 1978). See \cite{Tos} for the details.
In particular, if $N$ is a compact K\"{a}hler surface with a trivial $K_N$, then $N$ admits a \textit{hyperK\"{a}hler structure} since $SU(2) \cong Sp(1)$. Since any hyperK\"{a}hler manifold admits a a holomorphic symplectic form, so does $N$.
Now, we are ready to prove our main theorem.
\begin{theorem}[Theorem \ref{theorem : main}]
    There exists a compact K\"{a}hler manifold $(X,\omega,J)$ with $\dim_{\C} X = 3$ such that
    \begin{enumerate}
        \item $X$ is simply connected,
        \item $H^{2k+1}(X) = 0$ for every integer $k \geq 0$,
        \item $X$ admits a symplectic form $\sigma \in \Omega^2(X)$ of non-hard Lefschetz type, and
        \item $\sigma$ is deformation equivalent to the K\"{a}hler form $\omega$.
    \end{enumerate}
\end{theorem}

\begin{proof}

    We will construct $(X,\omega,J)$ with a projective $K3$-surface as follows.
    Consider a Fermat quartic $$N = \{ [z_0,z_1,z_2,z_3] \in \CP^3 ~| ~z_0^4 + z_1^4 + z_2^4 + z_3^4 = 0 \}.$$
    Obviously, $N$ is a projective $K3$-surface and there is a K\"{a}hler form $\phi \in \Omega^{1,1}(N)$ induced by the Fubini-Study form $\omega_{FS}$ on $\CP^3$ with $[\omega_{FS}] \in H^2(\CP^3,\Z)$. We denote by $e = [\phi] \in H^2(N,\Z)$ the K\"{a}hler class with respect to $\phi$.
    Since $N$ admits a hyperK\"{a}hler structure, there exists a holomorphic symplectic form $\rho \in \Omega^{2,0}$.
    Let $\sigma = \rho + \bar{\rho}$ be a real closed two form lying on $\Omega^{2,0} \oplus \Omega^{0,2}$.
    Since $\sigma^2 = 2\rho\wedge\bar{\rho} > 0$ everywhere on $N$,
    $\sigma$ is a symplectic form on $N$ and $\sigma \wedge \phi \in \Omega^{3,1} \oplus \Omega^{1,3}$ must vanish on $N$.

    Now, we will construct a family of symplectic forms $\{\sigma_t\}_{\-1 \leq t \leq 1}$ on $N$.
    By scaling $\sigma$, we may assume that $[\sigma]^2 = [\phi]^2 \in H^4(N,\Z)$.
    Let $$\sigma_t = \sigma - t \phi$$ for $-1 \leq t \leq 1$.
    Since $\sigma \wedge \phi \equiv 0$, we have $\sigma_t^2 = \sigma^2 + t^2\phi^2 > 0$ everywhere so that $\{ \sigma_t \}_{-1 \leq t \leq 1}$ is a family of symplectic forms such that $[\sigma_t] - [\sigma_s] = (s-t)[\phi] = (s-t)e$ for every $t,s \in [-1,1]$.
    Hence for $\delta > 0$, the corresponding manifold $(\widetilde{M}(N,e,1+\delta)_{-1,1}, \widetilde{\sigma})$ does not satisfy the hard Lefschetz property by Proposition \ref{proposition : if and only if}.

    From now on, let $X = \widetilde{M}(N,e,1+\delta)_{-1,1}$. By Proposition \ref{proposition : sphere bundle}, $X$ is diffeomorphic to
    $\mathbb{P}(\xi \oplus \underline{\C})$ where $\xi$ is a holomorphic line bundle over $N$ whose first Chern class is $e$.
    Since $e$ is chosen to be of type $(1,1)$ with respect to the complex structure on the $K3$-surface $N$, $X$ admits a K\"{a}hler structure by Proposition \ref{proposition : sphere bundle} again. To sum up, we constructed a compact manifold $X$ which admits a K\"{a}hler structure, and a symplectic form $\widetilde{\omega}_{\sigma}$ of non-hard Lefschetz type. Hence we proved (3) in Theorem \ref{theorem : main}.

    It remains to show that $X$ satisfies the conditions (1),(2), and (4) in Theorem \ref{theorem : main}.
    Since $N$ is simply connected and $X$ is a sphere bundle over $N$, (1) follows from the long exact sequence of homotopy groups of a fiber bundle.
    Also, (2) follows from the Leray-Serre spectral sequence by using the vanishing of cohomology of both $N$ and $S^2$ in even degrees.
    For (4), note that $\phi$ is deformation equivalent to $\sigma$ via the path of symplectic forms
    $ \{ t\sigma + (1-t)\phi \}_{0 \leq t \leq 1} $ since $\sigma \wedge \phi$ is identically zero on $N$.
    Hence $\widetilde{\omega}_{\sigma}$ is deformation equivalent to $\widetilde{\omega}_{\phi}$ by Lemma \ref{lemma : symplectic deformation}.
    Also, $\widetilde{\omega}_{\phi}$ is deformation equivalent to some K\"{a}hler form $\widetilde{\omega}_{\phi'}$ on $X$ for certain K\"{a}hler form $\phi'$ on $N$ by Proposition \ref{proposition : sphere bundle}. This completes the proof.

\end{proof}


\end{document}